\documentclass[a4paper]{amsart}
 \usepackage{amsfonts,amsmath,amssymb}
 \usepackage[english]{babel}
 \usepackage{indentfirst}
 \usepackage{bm}
 \usepackage{bbm}
 \usepackage{xcolor}
 \usepackage{amsthm}
 \usepackage{colortbl}
 \usepackage{enumerate}
 \usepackage{url}
 \usepackage[applemac]{inputenc}
 \usepackage[cm]{aeguill}
 \usepackage{graphicx}
 \usepackage{url}
 \usepackage{xcolor}
 \usepackage{fancybox}
 \usepackage{multicol}
 \usepackage{epsfig} 
 \usepackage{amsmath} 
 \usepackage{amssymb}  
 \usepackage{setspace}
\usepackage{tikz}

\newcommand{\R}{\mathbb{R}}
\newcommand{\NN}{\mathbb{N}}

\newcommand\C{\mathbb C}

\newcommand{\rst}[1]{\ensuremath{{\mathbin\upharpoonright}%
\raise-.5ex\hbox{$#1$}}}

\newtheorem{thm}{Theorem}

\newtheorem{lem}[thm]{Lemma}

\theoremstyle{definition}

\theoremstyle{remark}
\newtheorem{rem}[thm]{Remark}
\newtheorem{conjecture}[thm]{Conjecture}

\renewcommand{\theenumi}{\roman{enumi}.}



\newcommand{\N}{\mathbb{N}}

\begin{document}
  \title[]{On a nonlinear Schr\"odinger  equation for nucleons in one space dimension}

\author{Christian Klein}
\address{Institut de Math\'ematiques de Bourgogne, UMR 5584\\
                Universit\'e de Bourgogne-Franche-Comt\'e, 9 avenue Alain Savary, 21078 Dijon
                Cedex, France\\
    E-mail Christian.Klein@u-bourgogne.fr}

\author{Simona Rota Nodari}
\address{Institut de Math\'ematiques de Bourgogne, UMR 5584\\
                Universit\'e de Bourgogne-Franche-Comt\'e, 9 avenue Alain Savary, 21078 Dijon
                Cedex, France\\
    E-mail simona.rota-nodari@u-bourgogne.fr}
\date{\today}

\begin{abstract}
    We study a 1D nonlinear 
    Schr\"odinger equation appearing in the description of a particle 
    inside an atomic nucleus. For various nonlinearities, the 
    ground states are discussed and given in explicit form. Their stability is 
    studied numerically via the time evolution of perturbed ground 
    states. In the time evolution of general localized initial data, 
    they are shown to appear in the long time behaviour of certain 
    cases. 
\end{abstract}

 
\thanks{We thank the anonymous referees for useful comments and for a 
careful reading of the manuscript. This work is partially supported by 
the ANR-FWF project ANuI - ANR-17-CE40-0035, the ANR project DYRAQ - ANR-17-CE40-0016, the isite BFC project 
NAANoD, the ANR-17-EURE-0002 EIPHIby the 
European Union Horizon 2020 research and innovation program under the 
Marie Sklodowska-Curie RISE 2017 grant agreement no. 778010 IPaDEGAN 
and the EITAG project funded by the FEDER de Bourgogne, the region 
Bourgogne-Franche-Comt\'e and the EUR EIPHI}
\maketitle

\section{Introduction}


%

This paper is concerned with the study of solutions to a
nonlinear Schr\"odinger (NLS) type equation which, 
in a specific non-relativistic limit proper to nuclear physics, 
describes the behavior of a particle inside the atomic nucleus. This equation is, at least formally (see Appendix~\ref{formalderivation}), deduced from a relativistic model involving a Dirac operator and, in space dimension $d=1$, is given by 

\begin{equation}\label{eqmodtimealpha}
i\partial_t\phi=-\partial_x\left(\frac{\partial_x\phi}{1-|\phi|^{2\alpha}}\right)+\alpha|\phi|^{2\alpha-2}\frac{|\partial_x\phi|^2}{(1-|\phi|^{2\alpha})^2}\phi-a|\phi|^{2\alpha}\phi
\end{equation} 
where $\phi\in L^2(\R,\C)$ is a function that describes the quantum 
state of a nucleon (a proton or a neutron), $\alpha\in \N^*$ is a strictly 
positive integer and  $a>0$ is a parameter of the model. 
Note that equation (\ref{eqmodtimealpha}) is Hamiltonian and has a 
conserved energy
\begin{equation}
      E[\phi] = 
      \int_{\mathbb{R}}^{}\frac{|\partial_{x}\phi|^2}{1-|\phi|^{2\alpha}}-\frac{a}{\alpha+1}|\phi|^{2\alpha+2}.
 \label{E}
\end{equation}

Solitary wave solutions for this equation can be constructed by 
taking $\phi(t,x)=e^{ibt}\varphi(x)$ with $\varphi$ a real positive 
square integrable solution to the stationary equation 
\begin{equation}\label{eqmodstat}
-\left(\frac{\varphi'}{1-\varphi^{2\alpha}}\right)'+\alpha\frac{( \varphi')^2}{(1-\varphi^{2\alpha})^2}\varphi^{2\alpha-1}-a\varphi^{2\alpha+1}+b\varphi=0.
\end{equation}
The reasoning that the solution can be chosen to be real is the same as for the standard NLS equation.

Positive square integrable solutions of~\eqref{eqmodstat} can be seen as ground state solutions of~\eqref{eqmodtimealpha} since they are minimizers of~\eqref{E} among all the functions belonging to 
\begin{equation}
X=\left\{\varphi\in L^2(\R), \int_{\R}\frac{|\varphi'|^2}{(1-|\varphi|^{2\alpha})_+}<+\infty, \int_{\R} |\varphi|^2=1 \right\}\subset H^1(\R)
\end{equation}
where $f_+$ denotes the positive part of any function $f$. As shown in~\cite{EstRot-13}, this is the appropriate way to define ground states for this energy. Indeed, on the one hand, by adapting the arguments of~\cite{EstRot-13}, it can been shown that the energy $E$ is not bounded from below in the set $\left\{\varphi\in H^1(\R), \int_{\R} |\varphi|^2=1 \right\}$. On the other hand, if $\varphi\in X$, one can show (see~~\cite{EstRot-13}) that $|\varphi|^2\le 1$ a.e. in $\R$. As a consequence, for any $\varphi\in X$, $E[\varphi]\ge -\frac{a}{\alpha+1}$.

In this paper 
we will prove the existence of solitary wave solutions 
to~\eqref{eqmodtimealpha} for any value of $\alpha\in \N^*$, give 
them in explicit form, and study numerically their stability as well as the 
time evolution of more general initial data with $|\phi|<1$.

In the physical literature, the most relevant case is given by $\alpha=1$ which leads to the cubic nonlinear Schrödinger type equation

\begin{equation}\label{eqmodtime}
i\partial_t\phi=-\partial_x\left(\frac{\partial_x\phi}{1-|\phi|^2}\right)+\frac{|\partial_x\phi|^2}{(1-|\phi|^2)^2}\phi-a|\phi|^2\phi.
\end{equation}  

Nevertheless, it could be mathematically interesting to investigate also the 
behavior of solutions for other power nonlinearities as for example 
the quintic nonlinearity ($\alpha=2$) which corresponds to the $L^2$ 
critical case for the usual NLS equation. For the latter equation, it 
is known that initial data with a mass larger than the ground state 
can blow up in finite time, see for instance \cite{sul} and 
references therein.

To our knowledge, the above model was mathematically studied for the first time in \cite{EstRot-12}, where M.J. Esteban and S. Rota Nodari consider the equation
\begin{equation}\label{eqmodstatanydim}
-\nabla \cdot \left(\frac{\nabla \varphi}{1-\varphi^{2}}\right)+\frac{|\nabla \varphi|^2}{(1-\varphi^{2})^2}\varphi-a\varphi^{3}+b\varphi=0,
\end{equation}
which is the generalization of~\eqref{eqmodstat} for any spatial 
dimension $d\ge1$ and for $\alpha=1$. In particular, the existence of 
real positive radial square integrable solutions has been shown whenever $a>2b$. Note that solutions to~\eqref{eqmodstatanydim} do not have a 
simple scaling property in the parameter $b$ as ground states for the 
standard NLS equation. This makes it necessary to study several 
values of $b$ in this context. 

This result has then been generalized in \cite{TreRot-13}, where the 
existence of infinitely many square-integrable excited states 
(solutions with an arbitrary but finite number of sign changes) 
of~\eqref{eqmodstatanydim} was shown in dimension $d\ge 2$.


In \cite{EstRot-13} (see also \cite{LewRot-15}), using a variational 
approach the existence of solutions to~\eqref{eqmodstatanydim}  is proved 
without considering any particular ansatz for the  wave function of 
the nucleon and for a large range of values for the parameter $a$. Finally, in \cite{LewRot-15}, M. Lewin and S. Rota Nodari proved the uniqueness, modulo translations and multiplication by a phase factor, and the non-degeneracy of the positive solution to~\eqref{eqmodstatanydim}. The proof of this result is based on the remark that equation~\eqref{eqmodstatanydim} can be written in terms of $u=\arcsin(\varphi)$ as simpler nonlinear Schrödinger equation. 

The same can be done for~\eqref{eqmodstat}. Indeed, by taking $u:=\arcsin(\varphi^{\alpha})$, one obtains
\begin{equation}\label{eqmodstatu}
-u''-a\alpha\sin(u)^3\cos(u)+b\alpha\sin(u)\cos(u)
-\left(\frac{1}{\alpha}-1\right)(\partial_{x}u)^{2}\cot u=0.
\end{equation}

In Appendix~\ref{appexistencesolstat}, we generalize the results of~\cite{LewRot-15} for any $\alpha\in \N^*$ in spatial dimension $1$ by proving the following theorem. 

\begin{thm}\label{existence} Let $\alpha\in \N^*$. The nonlinear equation \eqref{eqmodstat} has no non-trivial solution $0\le \varphi <1$ such that $\lim_{x\to\pm\infty}\varphi(x)=0$ when $0<a\le (\alpha+1)b$. For $a>(\alpha+1)b>0$, the nonlinear equation \eqref{eqmodstat} has a unique solution $0<\varphi<1$ that tends to $0$ at $\pm \infty$, modulo translations. This solution is given by 
\begin{equation}\label{solstatexplicitthm}
\varphi(x)=\left(\frac{1}{2}\left(\frac{a}{(\alpha+1)b}+1\right)+\frac{1}{2}\left(\frac{a}{(\alpha+1)b}-1\right)\cosh({2\alpha\sqrt{b}x})\right)^{-\frac{1}{2\alpha}}.
\end{equation}
In particular, the following holds
\begin{enumerate}
\item $\varphi\in\mathcal C^{1}(\R)$;
\item $\varphi(x)=\varphi(-x)$;
\item $\varphi'(x)<0$ for all $x>0$;
\item $\varphi(x)\sim_{x\to+\infty}Ce^{-\sqrt{b}x}$;
\item $\varphi$ is non-degenerate.
\end{enumerate}
\end{thm}

The paper is organized as follows: In Section~\ref{secgroundstates} we derive the explicit form of solutions to~\eqref{eqmodstat} for any $\alpha\in \N^*$ whenever $a>(\alpha+1)b>0$, and we show their behavior for various values of the parameters. The computation presented in Section~\ref{secgroundstates} is justified in the Appendix~\ref{appexistencesolstat} where the proof of Theorem~\ref{existence} is done.
In Section 3, we outline the numerical approach for the 
time evolution of initial data according to (\ref{eqmodtimealpha}). 
This code is applied to perturbations of the ground states for 
various values of the nonlinearity parameter $\alpha$ and for initial 
data from the Schwartz class of rapidly decreasing functions. In Section~\ref{secoutlook}, we discuss the generalization of the model in higher space dimension. Finally, a formal 
derivation of the equation~\eqref{eqmodtimealpha} is presented in Appendix~\ref{formalderivation}. 

\section{Ground states}\label{secgroundstates}
In this section we construct ground state solutions to the 
equation (\ref{eqmodtimealpha}) and show some examples for different values of the parameters. 

First of all, equation (\ref{eqmodstat}) can be integrated once to give
\begin{equation}
    -\frac{(\varphi')^{2}}{1-\varphi^{2\alpha}}-\frac{a}{\alpha+1}\varphi^{2\alpha+2}+b\varphi^{2}=0
    \label{eqmodstatint},
\end{equation}
where we have used the asymptotic behavior of $\varphi$ for 
$x\to\infty$. Putting $\psi:=\varphi^{-2\alpha}$, we get from 
(\ref{eqmodstatint}), 
\begin{equation}
    (\psi')^{2}=4\alpha^{2}b\left(\psi-\frac{a}{(\alpha+1)b}\right)(\psi-1)
    \label{chieq},
\end{equation}
which has for $a\neq(\alpha+1)b$ the solution
\begin{equation}
    \psi(x)=\frac{1}{2}\left(1+\frac{a}{(\alpha+1)b}\right)+
    \frac{1}{2}\left|1-\frac{a}{(\alpha+1)b}\right|\cosh(2\alpha 
    \sqrt{b}(x-x_{0}))
    \label{chisol}.
\end{equation}
Here $x_{0}$ is an integration constant reflecting the translation 
invariance in $x$ of the ground state and $\psi(x_0)$ is chosen in 
order to have a $\mathcal C^1$ solution to~\eqref{chieq} defined for any $x\in\R$. Using the translation invariance, we will assume in the following that the maximum of the solution is at $x=0$, and then we put $x_{0}=0$. 

The solution to equation (\ref{chieq}) for $a=(\alpha+1)b$ leading to the 
wanted asymptotic behavior of $\varphi$ will not be globally 
differentiable. 

For $a<(\alpha+1)b$, one has $\psi(0)=1$ and thus 
$\varphi(0)=1$. This would lead to a vanishing denominator in 
(\ref{eqmodstatint}). As a consequence, $\varphi'(0)$ has to be equal to $0$. This contradicts equation~\eqref{eqmodstatint} since $a<(\alpha+1)b$.

 
Summing up, with (\ref{chisol}) we get the ground states for 
$0<(\alpha+1)b<a$ in the form 
\begin{equation}
    \varphi (x)= \left[\frac{1}{2}\left(1+\frac{a}{(\alpha+1)b}\right)+
    \frac{1}{2}\left(\frac{a}{(\alpha+1)b}-1\right)\cosh(2\alpha 
    \sqrt{b}x)\right]^{-\frac{1}{2\alpha}}.
    \label{gsexplicit}
\end{equation}

Let us point out that this construction will be further justified in Appendix~\ref{appexistencesolstat} where Theorem~\ref{existence} is proven.

As a concrete example we show the  solutions (\ref{gsexplicit})  for 
$a=9$ and various values of $b<a/(\alpha+1)$. 
The solutions for $\alpha=1$ can be seen in Fig.~\ref{figGS}. 
With $b\to a/2$, the solutions become broader and broader and have a 
larger maximum. The peak near 1 becomes also flatter. For $b=4.499$, 
the maximum is roughly at $0.9999$ and almost touches on some 
interval the line 1. 

\begin{figure}[htb!]
  \includegraphics[width=0.8\textwidth]{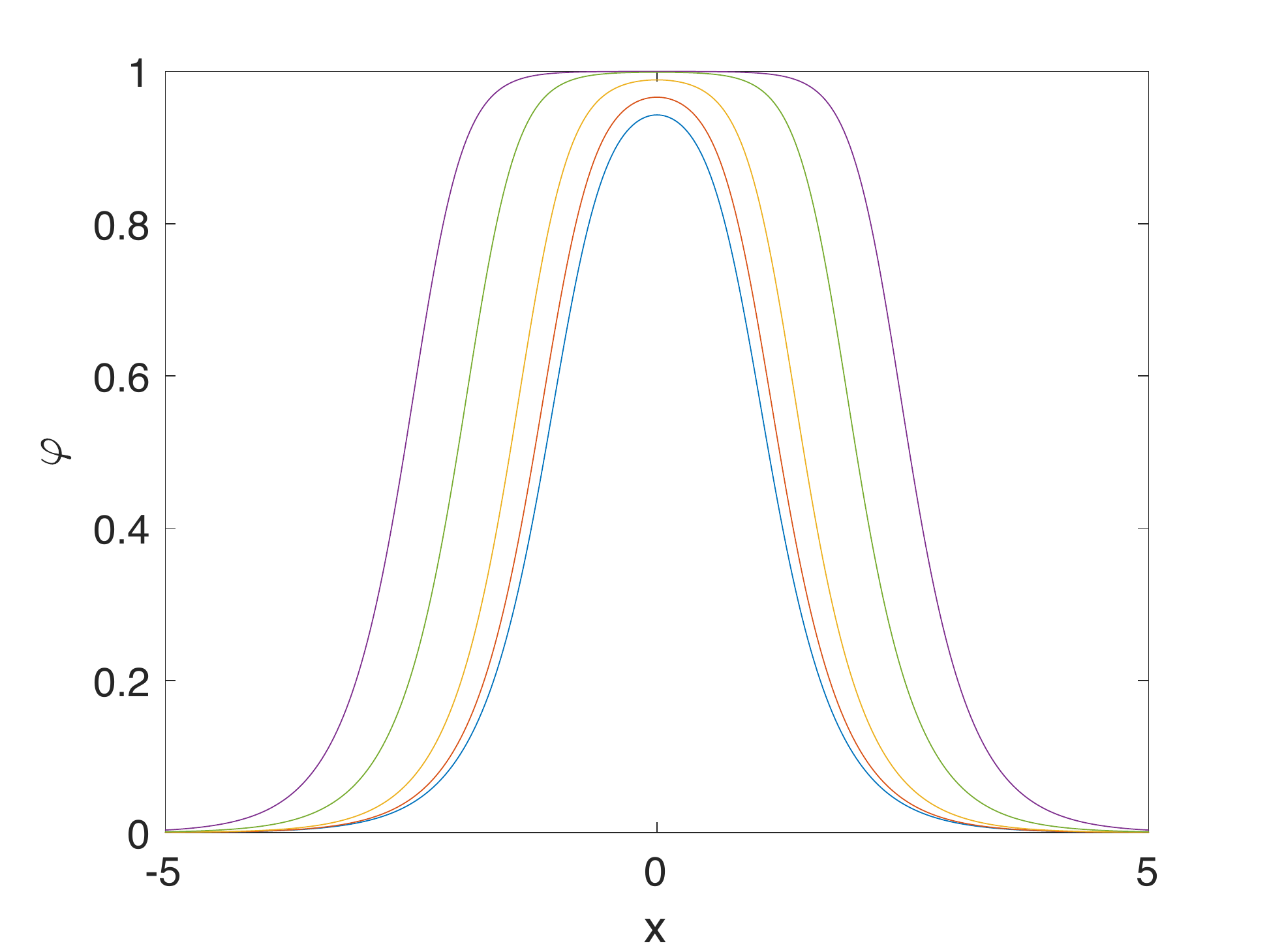}
  \caption{Ground state solution to equation (\ref{eqmodstat}) for 
  $\alpha=1$, $a=9$ and $b=4.0$, $4.2$, $4.4$, $4.49$ and $4.499$ (from bottom to 
  top). }
 \label{figGS}
\end{figure}

For 
$\alpha=2,3$ we get in the same way the figures in Fig.~\ref{figGS23}. It can 
be seen that the higher nonlinearity has a tendency to lead to more 
compressed peaks as in \cite{AKS}. But due to the missing scaling 
invariance of the ground states here, it is difficult to compare them. 
\begin{figure}[htb!]
  \includegraphics[width=0.49\textwidth]{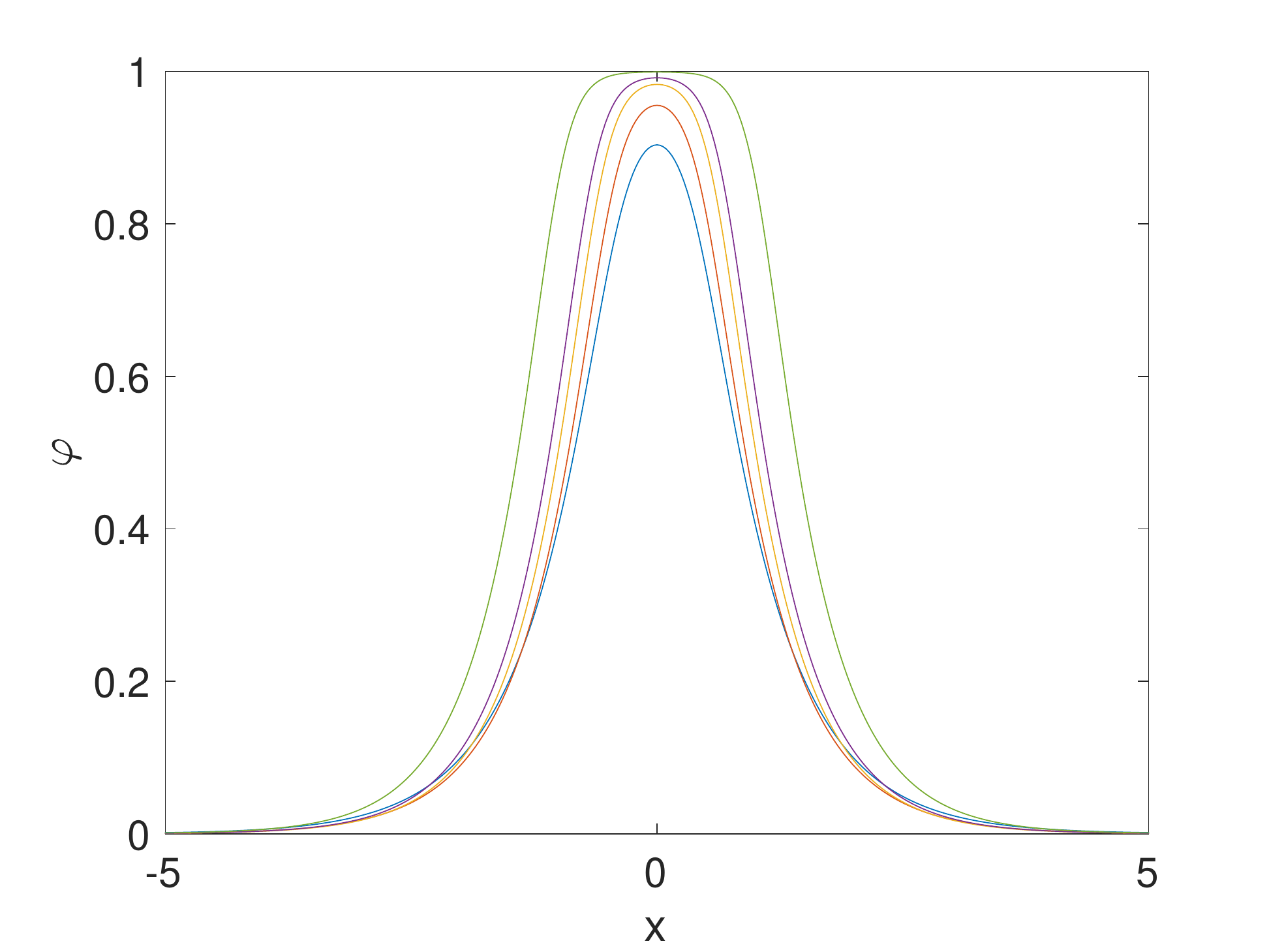}
  \includegraphics[width=0.49\textwidth]{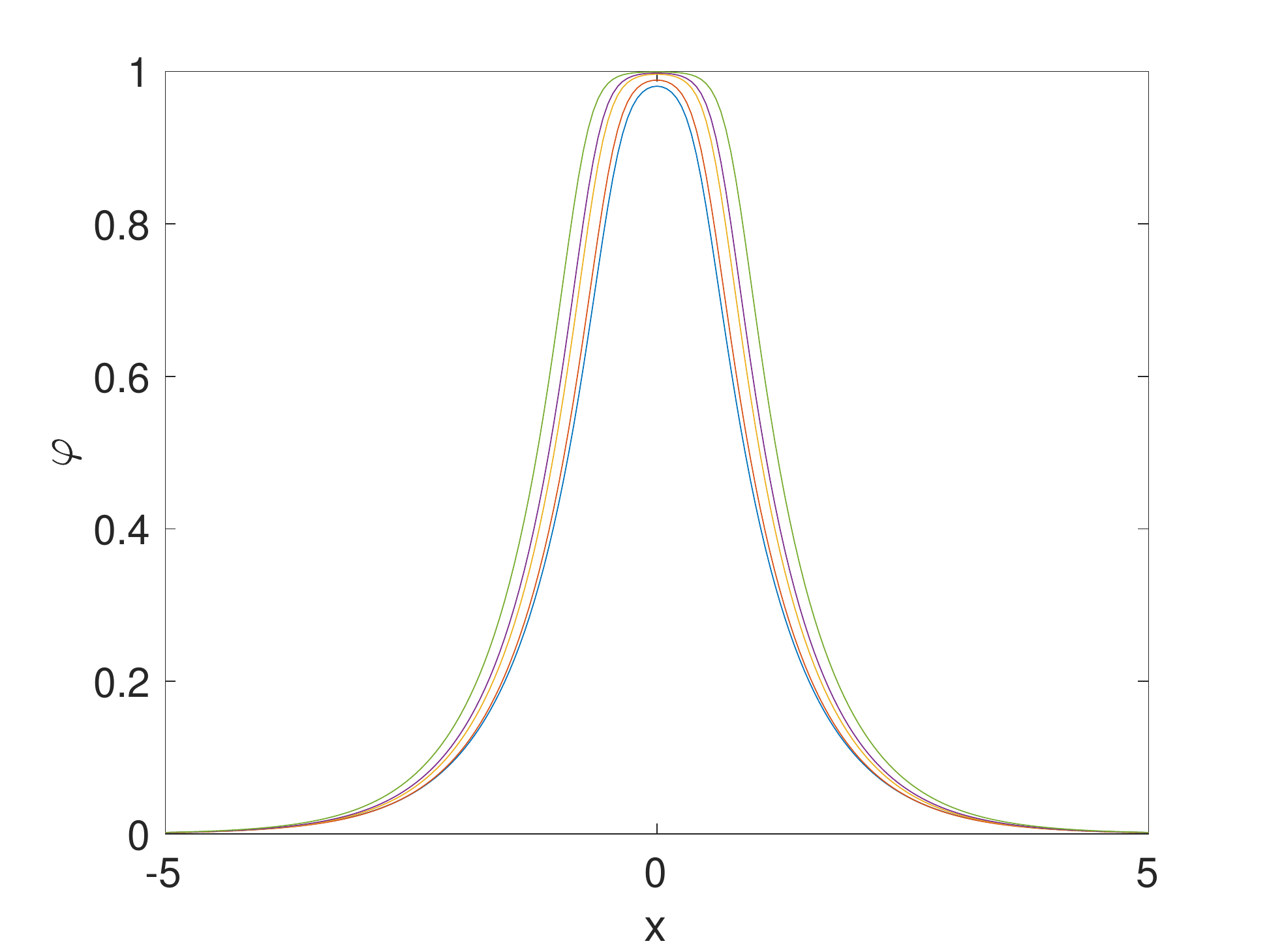}
  \caption{Ground state solution to equation (\ref{eqmodstat}) for 
  $a=9$ and $\alpha=2$ on the left  ($b=2.0$, $2.5$, $2.8$, $2.9$ 
  and $2.99$ from bottom to 
  top), and for $\alpha=3$ on the right ($b=2.0$, $2.1$, $2.2$, 
  $2.22$ and $2.24$ from bottom to 
  top). }
 \label{figGS23}
\end{figure}

\section{Numerical study of the time evolution}
In this section we study the time evolution of initial data for the 
equation (\ref{eqmodtimealpha}). We study the stability of the ground 
state 
and the time evolution of general initial data in the Schwartz class 
of smooth rapidly decreasing functions for various parameters. 

The results of this section can be summarized in the following 
\begin{conjecture}
    The ground states of equation (\ref{eqmodtimealpha}) are 
    asymptotically stable if the perturbed initial data satisfy 
    $|\phi(x,0)|<1$. The long time behavior of solutions for general 
	localized initial data is characterized by ground states and 
	radiation. 
\end{conjecture}

Note that the separation of radiation from the bulk is much faster 
for the cubic case. For higher nonlinearity, this takes considerably 
longer, and on the used computers, we can reach an agreement of the 
final state and a ground state of the 
order of a few percents.

\subsection{Time evolution approach}
The exponential decay of 
the stationary solutions  makes the use of Fourier spectral methods attractive. 
Thus we define the standard Fourier transform of a function $u$, 
\begin{equation}
        \hat{u}  = \mathcal{F}u:=\frac{1}{2\pi}\int_{\mathbb{R}}^{} 
    e^{-ikx}u\, dx,\quad k\in\mathbb{R},
    \label{fourier}
\end{equation}
and consider the $x$-dependence in equation (\ref{eqmodtimealpha}) in 
Fourier space. 

The numerical solution is constructed on the interval $x\in L[-\pi,\pi]$ where 
$L>0$ is chosen such that the solution and relevant derivatives 
vanish with numerical precision (we work here with double precision 
which corresponds to an accuracy of the order of $10^{-16}$). The 
solution $\phi$ is approximated via a truncated Fourier series where the 
coefficients $\hat{\phi}$ are 
computed efficiently via a \emph{fast Fourier transform} (FFT).  
 This means we treat the equation,
\begin{equation}\label{eqmodtimefourier}
i\partial_t \hat{\phi}= 
-ik\mathcal{F}\left(\frac{\partial_{x}\phi}{1-|\phi|^{2\alpha}}\right)+\mathcal{F}\left(\alpha|\phi|^{2\alpha-2}\frac{|\partial_x\phi|^2}{(1-|\phi|^{2\alpha})^2}\phi-a|\phi|^{2\alpha}\phi\right),
\end{equation}
and approximate the Fourier transform in (\ref{eqmodtimefourier}) by 
a discrete Fourier transform. 

The study of the  solutions to (\ref{eqmodtimefourier}) is 
challenging for several reasons: first it is an NLS equation which 
leads to a \emph{stiff} system of ODEs if FFT techniques  are used. 
Since a possible definition of \emph{stiffness} is that explicit 
time integration schemes are not efficient, the use of special 
integrators is recommended in this case. But most of the explicit 
stiff integrators for NLS equations, see for instance \cite{etna} and 
references therein, assume a stiffness in the linear part of the 
equations. However, here the second derivatives with respect to $x$ 
appear in nonlinear terms. Note that the NLS equation is not stiff if 
perturbations of the ground states are considered, but mainly when 
zones of rapid modulated oscillations appear, so called dispersive 
shock waves, see the discussion in \cite{etna}. 

Since we are interested in the former, the main problem of equation 
(\ref{eqmodtimealpha}) is not the stiffness, but the singular 
term for $|\phi|\to 1$. Since the equation is focusing, it is to be 
expected that for initial data with modulus close to 1 it will be 
numerically challenging since the focusing nature of the equation 
might lead for some time to even higher values of $|\phi|$. Obviously 
the regime $\phi\sim 1$ is the most interesting from a mathematical 
point of view since here the strongest deviation from the standard 
NLS equation is to be expected. To achieve the needed high accuracy 
for this,  a fourth order 
method is necessary, and even there with small time steps. It turns 
out that in the studied examples the requirement for accuracy is of 
similar order as stability conditions for  an explicit approach. We 
apply here  the standard 
explicit fourth order Runge-Kutta method for which the stability 
condition is $\Delta t\sim 1/N^{2}$, see for instance 
\cite{trefethen}. We have compared this approach to the 
unconditionally stable second order Crank-Nicolson method, but had 
trouble to reach the needed accuracy in an efficient way. The 
explicit approach is also 
more efficient than an implicit 4th order Runge-Kutta scheme as 
applied in \cite{etna} where a nonlinear equation has to be solved iteratively at 
each time step. 

The accuracy 
of the solution is controlled as in \cite{etna}: the decrease of 
the Fourier coefficients indicates the spatial resolution since the 
numerical error of a truncated Fourier series is of the order of the 
first neglected Fourier coefficients. The error in the time 
integration is controlled via conserved quantities. We use the energy 
(\ref{E})
which is a conserved quantity of (\ref{eqmodtimealpha}), but which will 
numerically depend on time due to unavoidable numerical errors. In 
the examples below, the relative energy is always conserved to better 
than $10^{-6}$. 

We test the numerical approach at the example of the ground state. 
Concretely we consider the ground state solution for $\alpha=1$, $a=9$, $b=4.4$ as 
initial data. We use $N=2^{10}$ Fourier modes in $x$ for 
$x\in5[-\pi,\pi]$ and $N_{t}=10^{5}$ time steps for $t\in[0,1]$. 
Note that though the ground state solution is stationary, it is not 
time independent. We compare the numerical and the exact solution, 
i.e., the solution to (\ref{gsexplicit}) times $e^{ibt}$ at 
the final time $t=1$. This difference  is of the order of 
$10^{-14}$ as shown in Fig.~\ref{figdelta}. The relative conservation of the energy (\ref{E}) is 
during the whole computation of the order of $10^{-14}.$ This shows 
that the ground state can be numerically evolved with an 
accuracy of the order $10^{-14}$, and that the conservation of the 
numerically computed energy 
indicates the accuracy of the time integration. 
\begin{figure}[htb!]
  \includegraphics[width=0.5\textwidth]{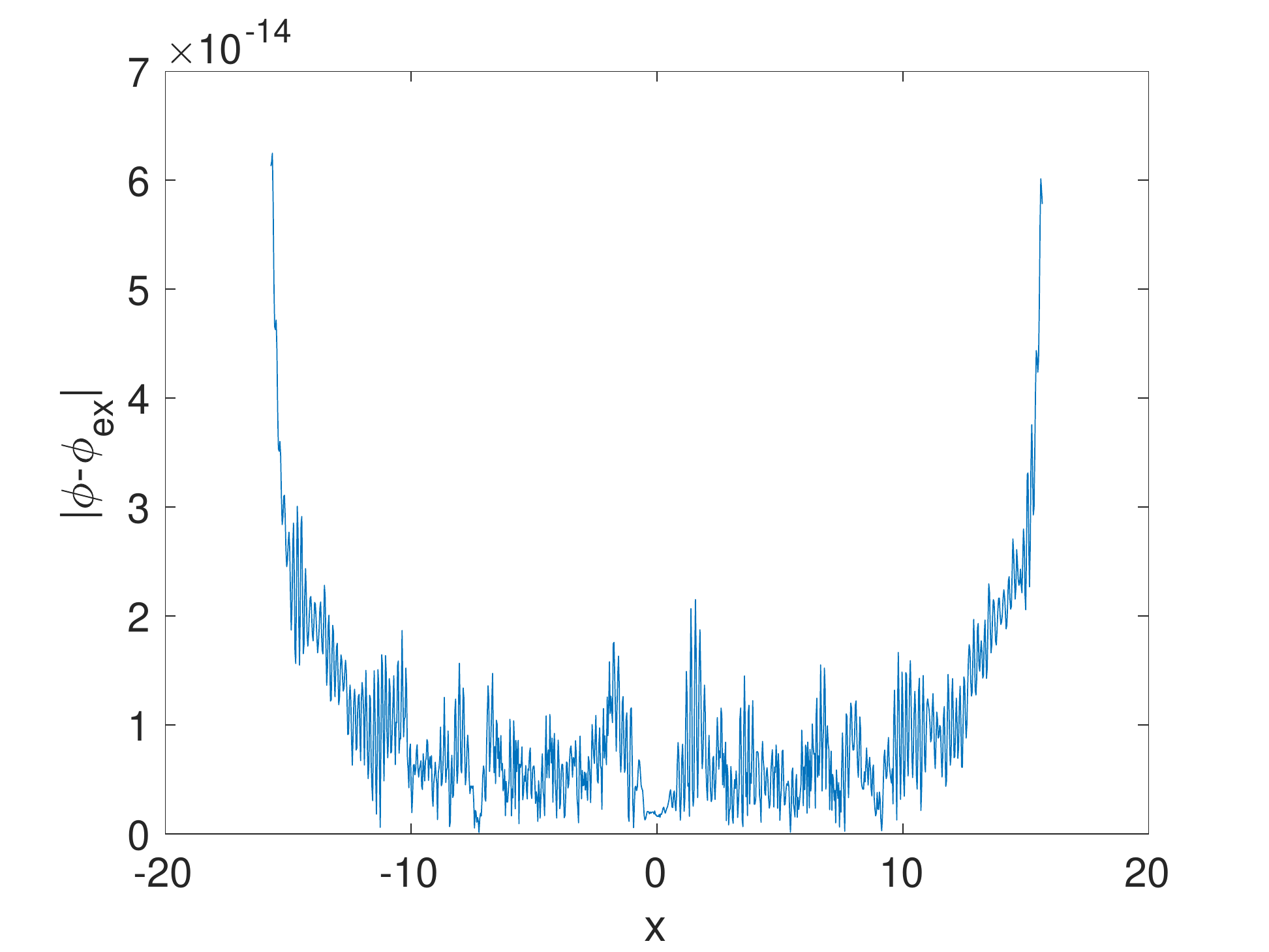}
  \caption{Difference of the numerical solution to the equation (\ref{eqmodtime}) for initial 
  data being the ground state solution  for 
  $a=9$ and $b=4.4$, and the exact solution for $t=1$. }
 \label{figdelta}
\end{figure}

\subsection{Perturbations of ground states}

In this subsection we consider the stability of the ground states 
(\ref{gsexplicit}). To this end we perturb it first in the form 
$\phi(x,0)=\lambda \varphi(x)$, where $\lambda\sim 1$. 
\begin{rem}\label{rem1}
	Numerically one cannot consider arbitrary small perturbations as in 
analytical work since one would have to wait for very large times 
in order to get meaningful results. But using long times would imply 
that numerical errors of even high order schemes pile up. Thus in 
practice one always considers perturbations of the order of $1\%$ 
(some equations like the Korteweg-de Vries equation allow 
perturbations of the order of $10\%$ such that the solution stays 
close to a soliton, and the present equation is similar in this 
respect). This implies, however, that the final state of a perturbed 
ground state is not the exact ground state  even for 
asymptotically stable ground states, but a nearby one. 

\end{rem}

\paragraph{\textbf{The cubic case}}

We use $N = 2^{11}$ 
Fourier modes and $N_{t}=5*10^{5}$ time steps for $t\in[0,0.25]$, 
i.e., more than a whole period of the perturbed ground state.  In 
Fig.~\ref{figGS9.44.099} we show the solution for the perturbed 
ground state with $\lambda=0.99$. It can be seen that after a 
short phase of focusing a ground state with slightly larger maximum 
than the initial data 
is reached. In addition there is some radiation towards infinity. 
The Fourier coefficients of the solution 
at the final time on the right of 
Fig.~\ref{figGS9.44.099} indicate that the solution is fully resolved 
in $x$. 
\begin{figure}[htb!]
  \includegraphics[width=0.49\textwidth]{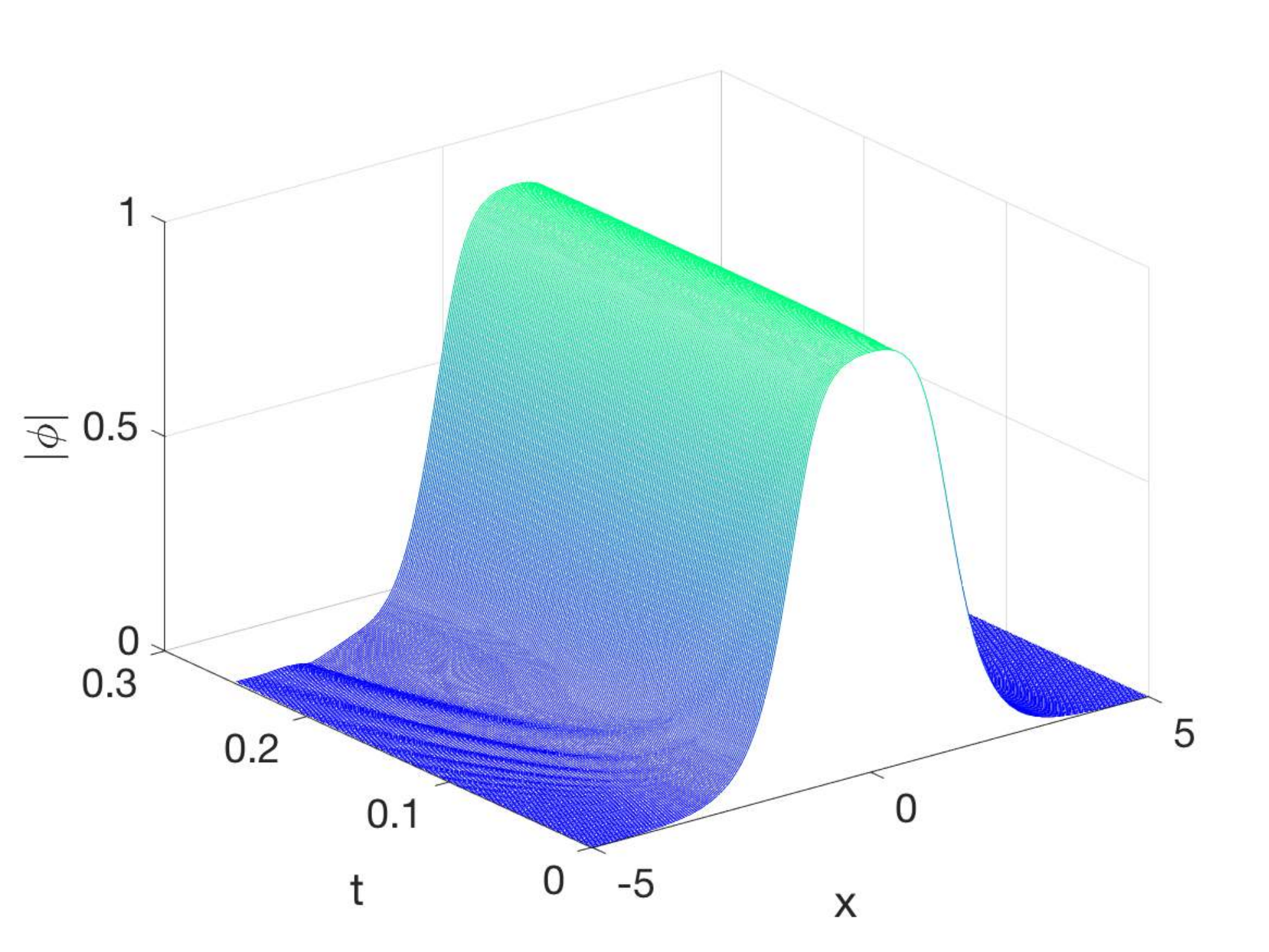}
  \includegraphics[width=0.49\textwidth]{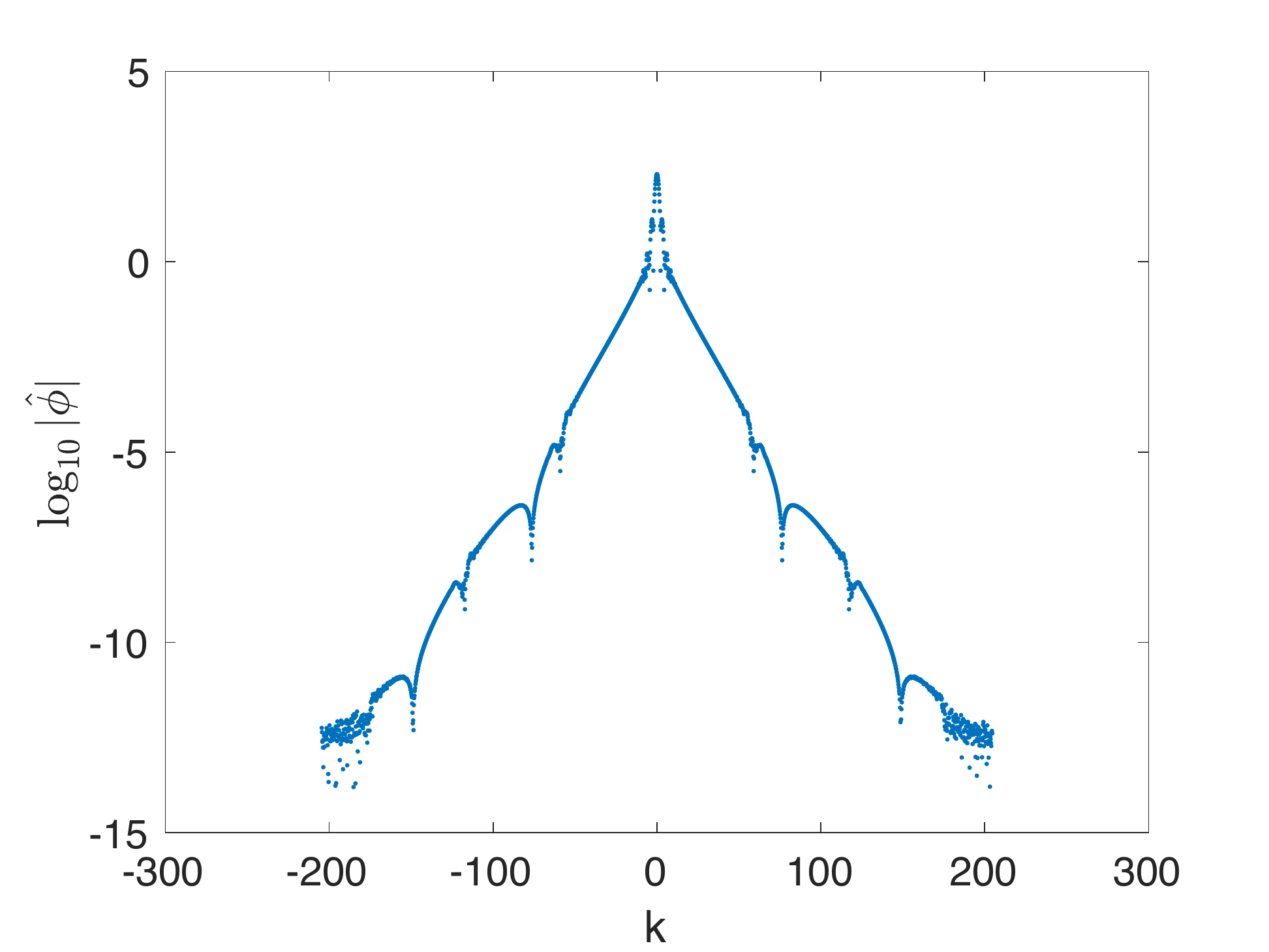}
\caption{Solution to the equation (\ref{eqmodtime}) for the initial 
  data $\phi(x,0)=0.99 \varphi(x,a=9,b=4.4)$ and $\alpha=1$ on the 
  left, and the Fourier coefficients of the solution at the final 
  time on the right. }
 \label{figGS9.44.099}
\end{figure}

The reaching of a ground state is also suggested by the $L^{\infty}$  
norm of the solution shown on the left of 
Fig.~\ref{figGS9.44.099inf}.  As stated in remark \ref{rem1}, we 
expect a ground state of slightly different $b$ as the final state 
since we have a perturbation of the order of $1\%$. To determine this 
ground state, we use some optimal fitting of the ground state 
(\ref{gsexplicit}) by varying $b$ in order to approximate the solution at the final time. 
To this end we minimize the residual of the modulus of $\phi$ and $\varphi$ 
for $|x|<x_{0}$ (we consider $x_{0}=5$) via the optimization 
algorithm \cite{fminsearch} implemented in Matlab as the function 
\emph{fminsearch}. On the right of 
Fig.~\ref{figGS9.44.099inf}, we show the solution of 
Fig.~\ref{figGS9.44.099} at the final time in blue together with a fitted 
ground state in green. The good 
agreement (the green curve covers the blue one in the plot where it 
is identical up to plotting accuracy) shows that the final state is indeed a very nearby ground 
state, $b = 4.388$, which can be already indentified (the difference is of the 
order of $10^{-3}$) at an early time. 
\begin{figure}[htb!]
  \includegraphics[width=0.45\textwidth]{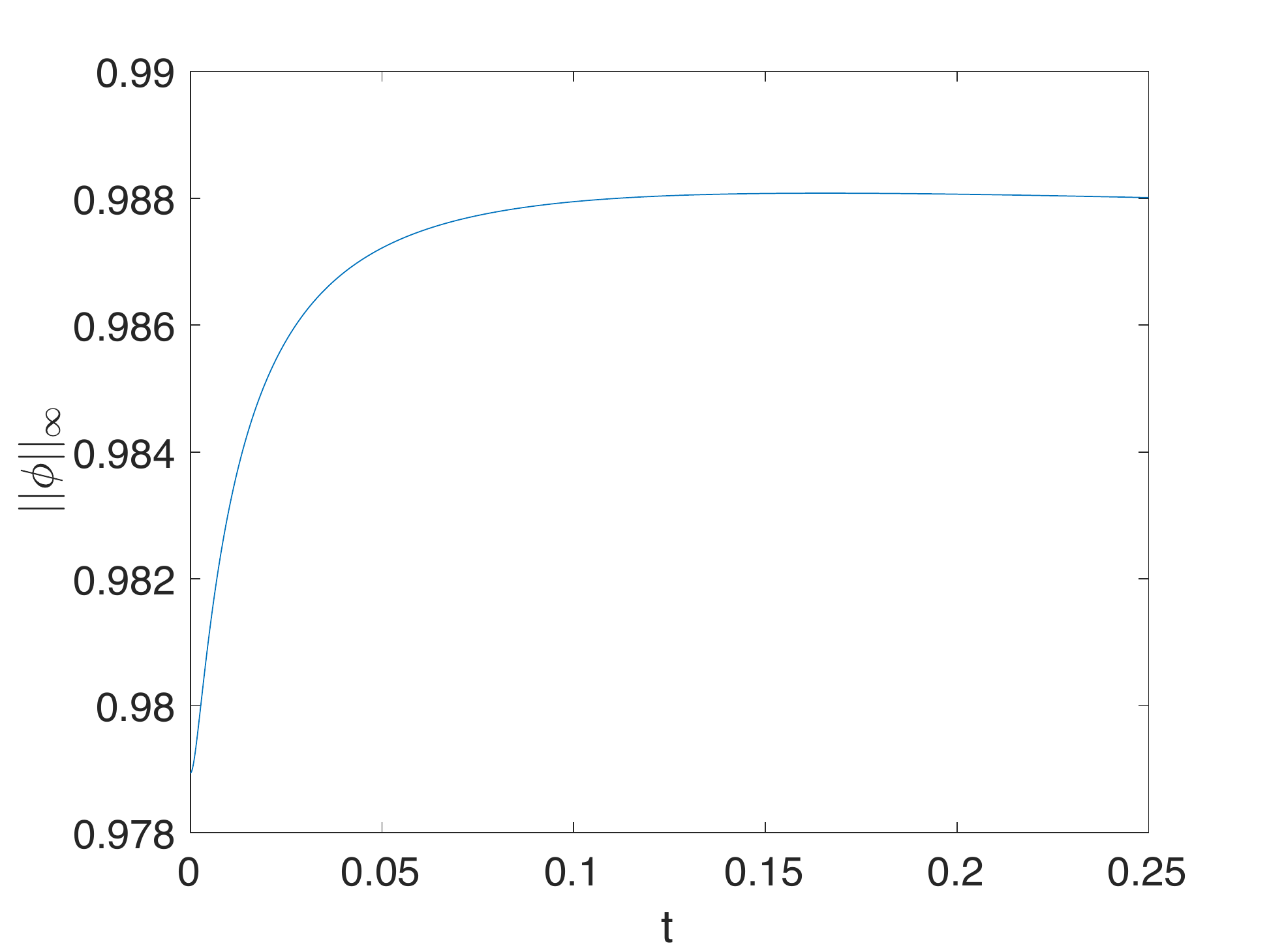}
  \includegraphics[width=0.45\textwidth]{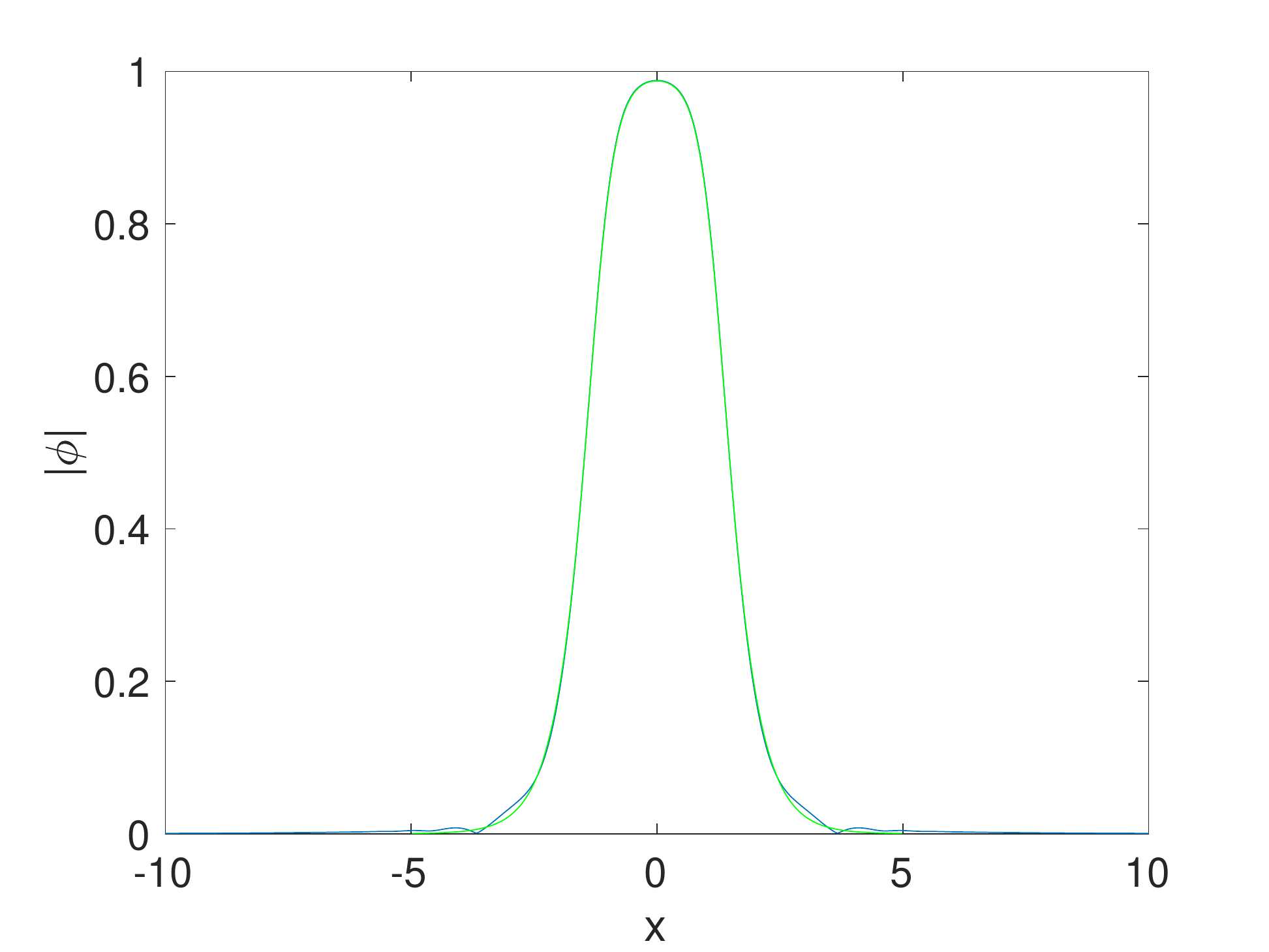}
  \caption{On the left the $L^{\infty}$ norm of the solution of 
  Fig.~\ref{figGS9.44.099}, on the right the solution at the final time in blue and a fitted 
  ground state in green.  }
 \label{figGS9.44.099inf}
\end{figure}

If we perturb the same ground state as in Fig.~\ref{figGS9.44.099} 
with a factor $\lambda>1$ (such that $||\lambda \varphi||_{\infty}<1$), 
we observe a similar behavior as can be seen in Fig.~\ref{figGS9.44.1001}.
The decrease of the modulus of the 
Fourier coefficients on the right of Fig.~\ref{figGS9.44.1001} 
indicates that the numerical error in the spatial resolution is of 
the order of $10^{-8}$. This shows that there are   stronger 
gradients to resolve in this case than in Fig.~\ref{figGS9.44.099}. 
\begin{figure}[htb!]
  \includegraphics[width=0.49\textwidth]{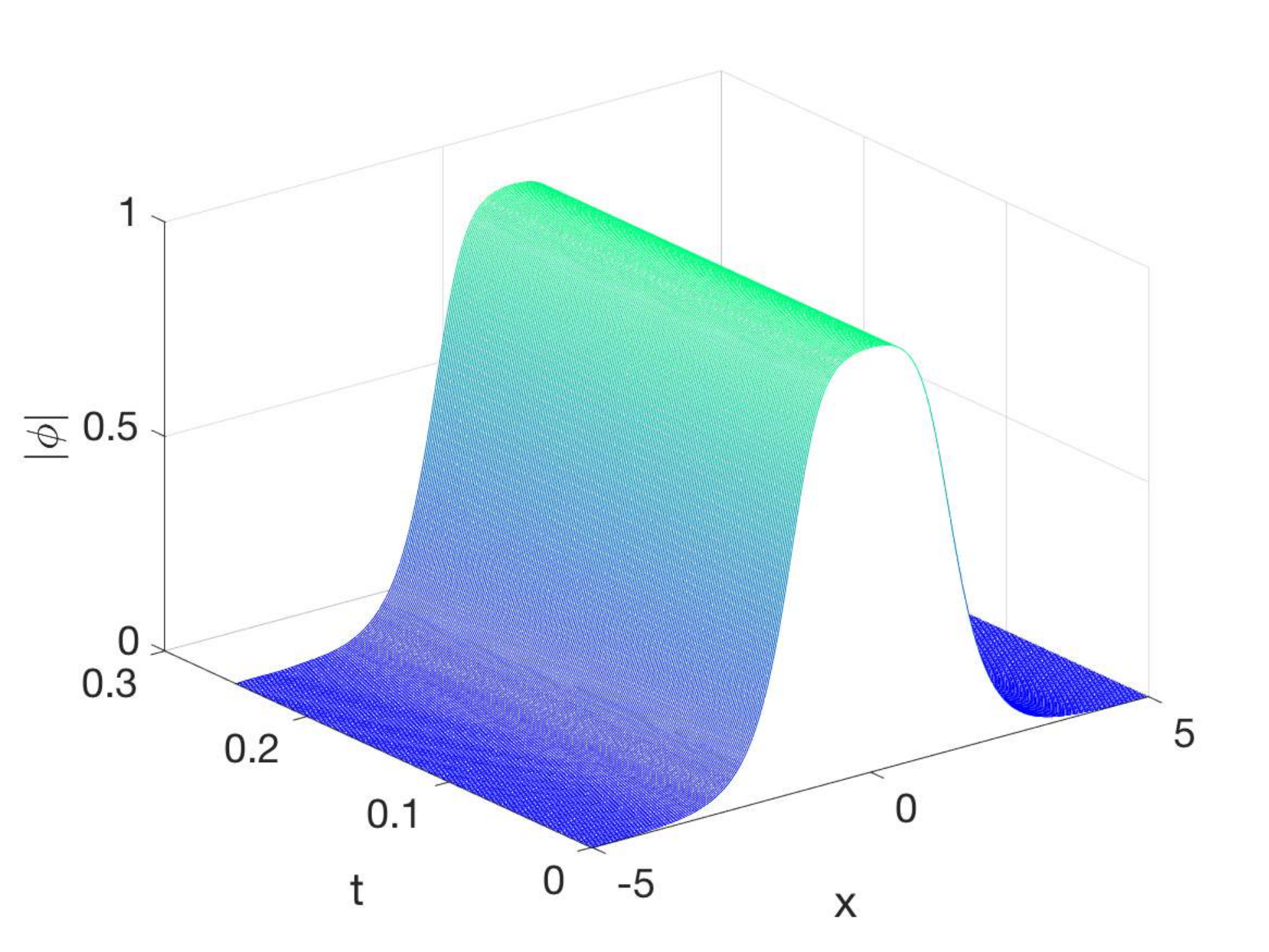}
  \includegraphics[width=0.45\textwidth]{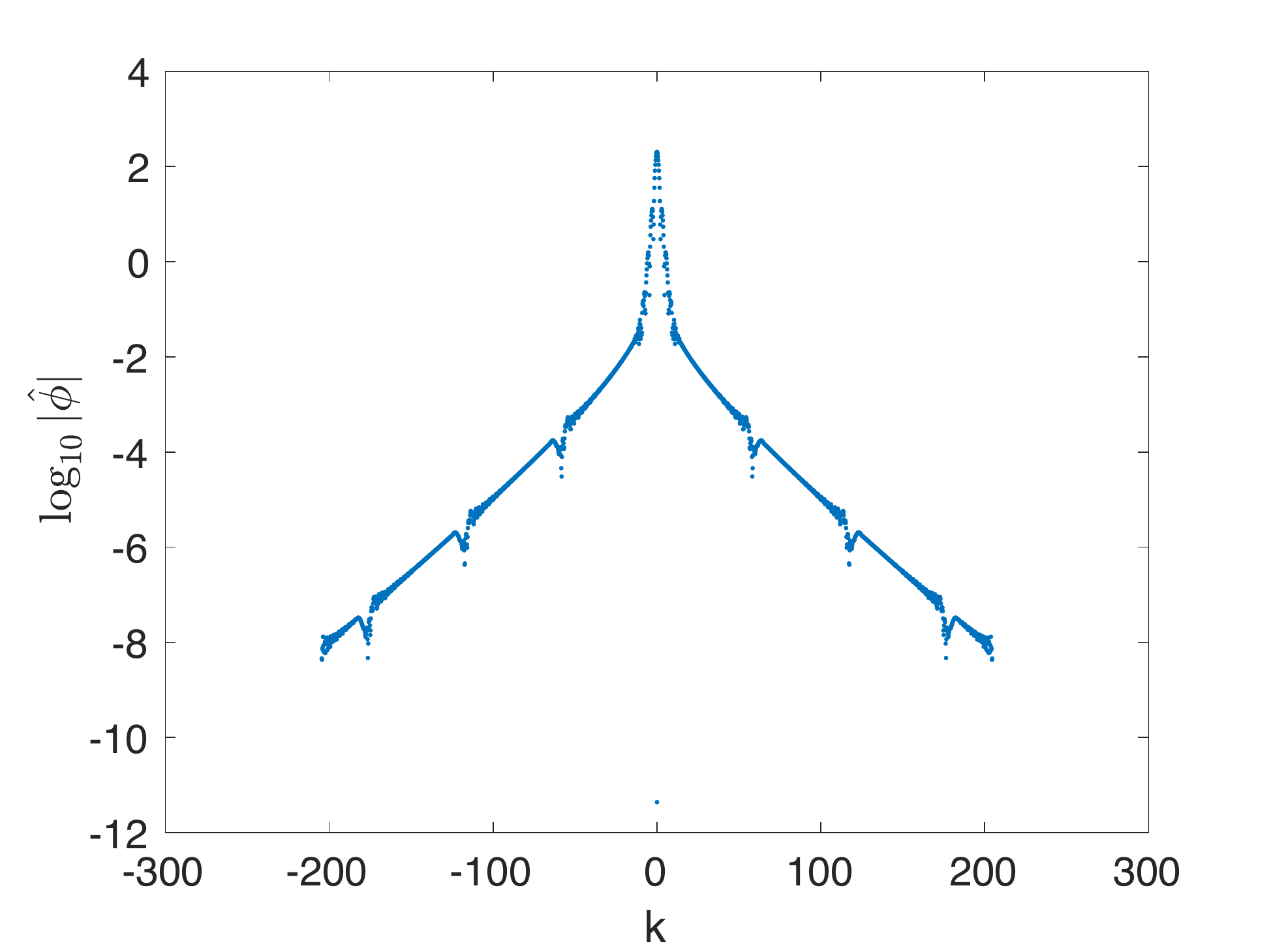}
  \caption{Solution to the equation (\ref{eqmodtime}) for the initial 
  data $\phi(x,0)=1.001\varphi(x,a=9,b=4.4)$ on the left and the 
  Fourier coefficients of the solution at the final time on the right.}
 \label{figGS9.44.1001}
\end{figure}

As is more obvious from the $L^{\infty}$ norm in Fig.
\ref{figGS9.44.1001inf} on the left, a ground state with slightly 
lower maximum than the initial data is quickly reached.  On the right 
of the same figure we show the solution at the final time in blue together 
with a fitted ground state ($b=4.4011$) in green. The agreement is so 
good that the blue line can hardly be seen (the difference is of the 
order of $10^{-3}$). 
\begin{figure}[htb!]
  \includegraphics[width=0.45\textwidth]{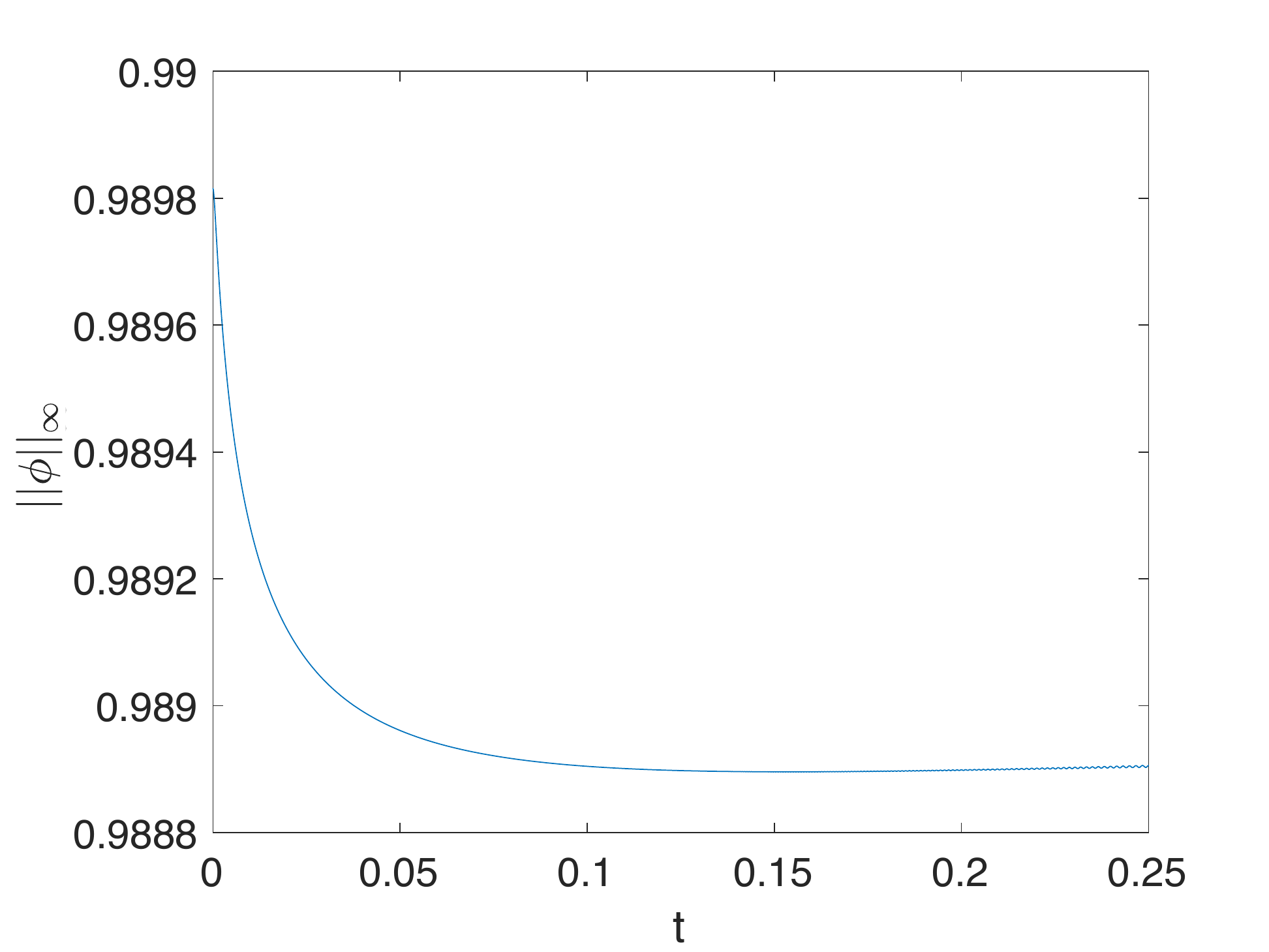}
  \includegraphics[width=0.45\textwidth]{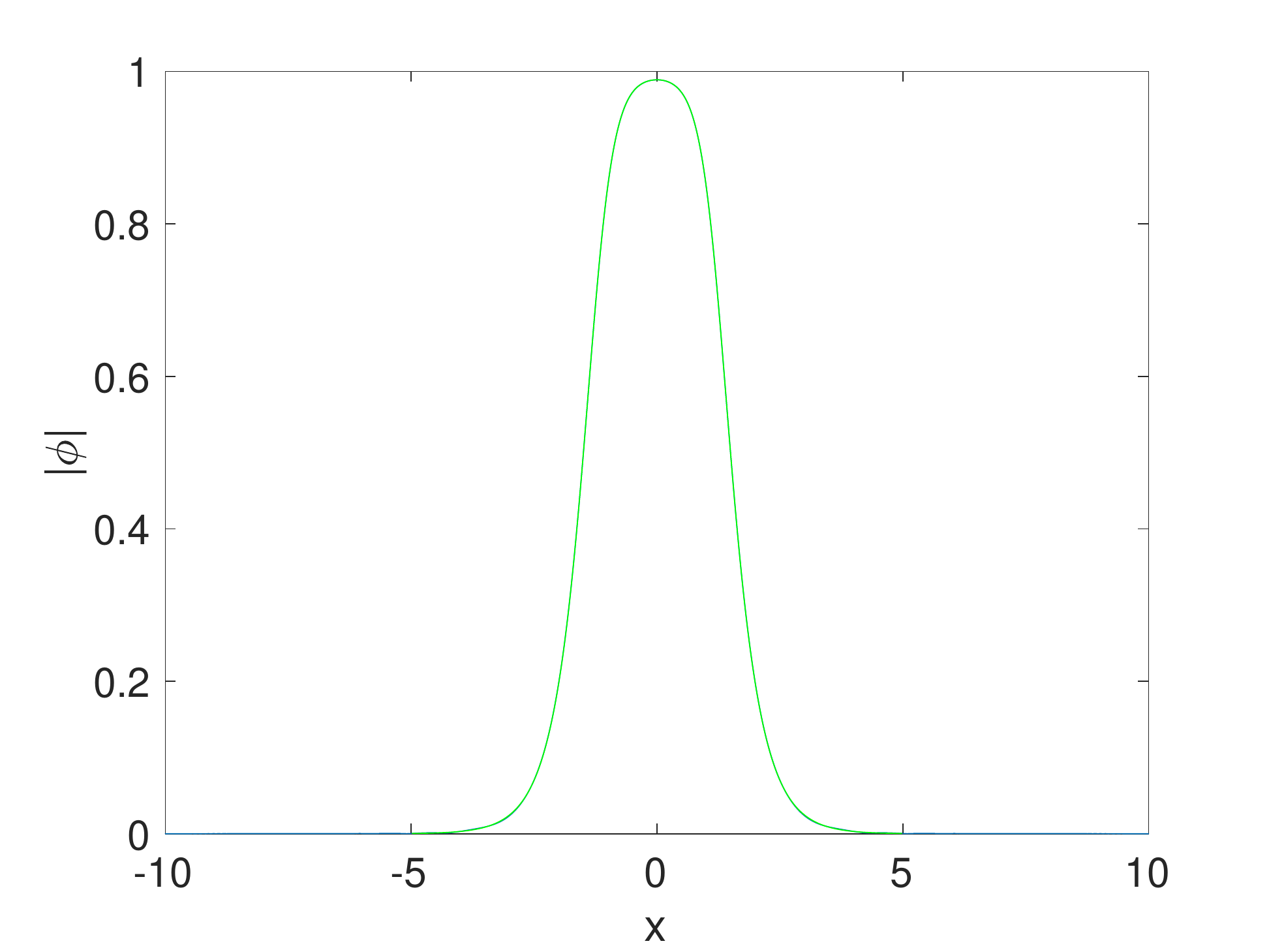}
  \caption{On the left the $L^{\infty}$ norm of the solution of 
  Fig.~\ref{figGS9.44.099}, on the right the solution at the final time in blue and a fitted 
  ground state in green.  }
 \label{figGS9.44.1001inf}
\end{figure}

The same initial data as above are perturbed with a localized 
perturbation of the form $\phi(x,0)=\varphi(x)\pm 0.001\exp(-x^{2})$. 
The resulting $L^{\infty}$ norms of the solutions to 
(\ref{eqmodtime}) for these initial data are shown in 
Fig.~\ref{NLSGSa9b44m001gaussinf}. In both cases the $L^{\infty}$ 
norm appears to approach a slightly smaller ground state than the 
unperturbed one (for the $-$ sign 
in the initial data) respectively slightly larger ground state (for the 
$+$ sign). Note that in the former case, the $L^{\infty}$ norm grows 
monotonically from its initial value to a value slightly below 
$0.9888$, whereas it decreases in the latter case monotonically from 
its initial value to a value slightly above $0.9888$. 
\begin{figure}[htb!]
  \includegraphics[width=0.45\textwidth]{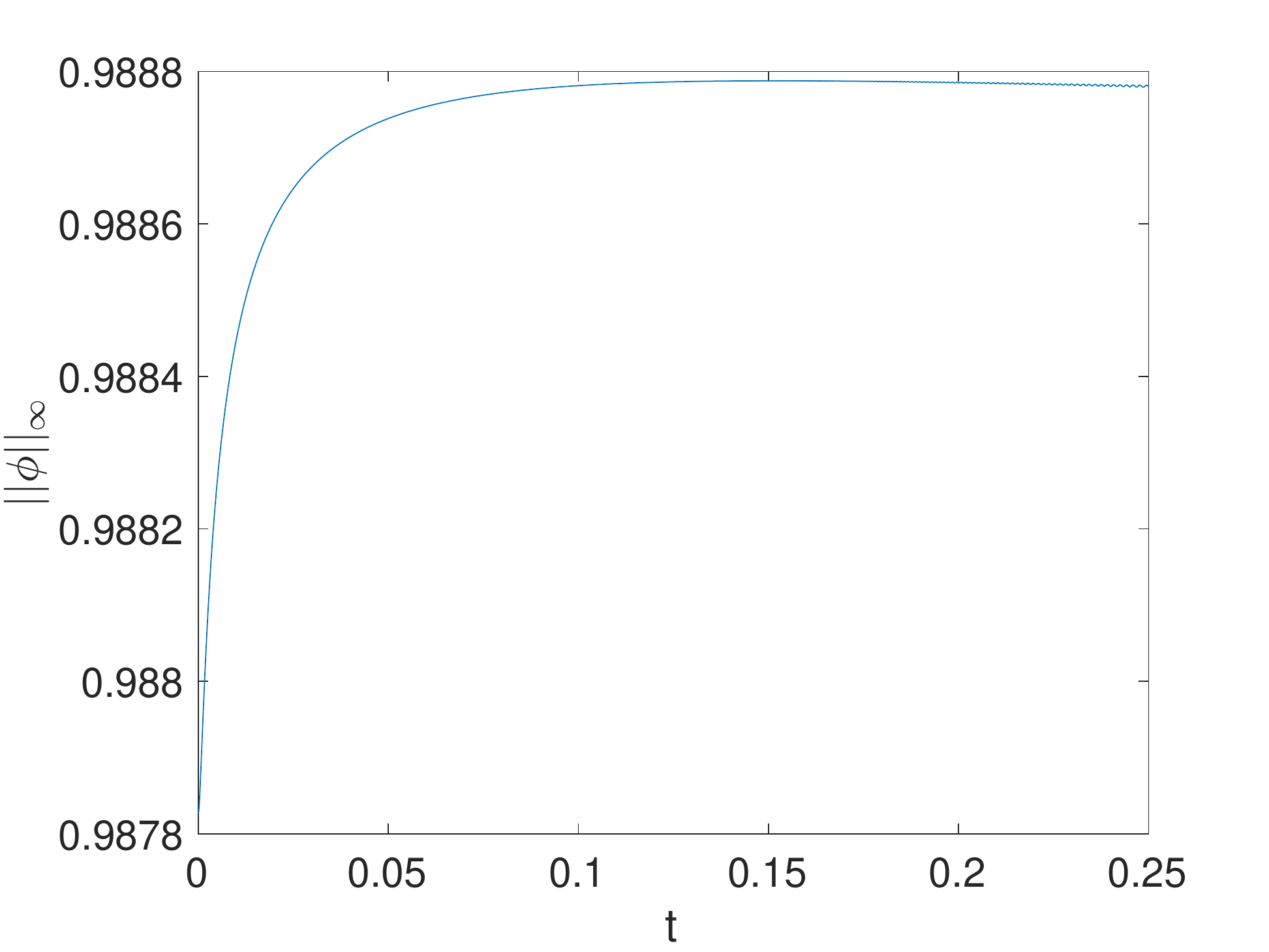}
  \includegraphics[width=0.45\textwidth]{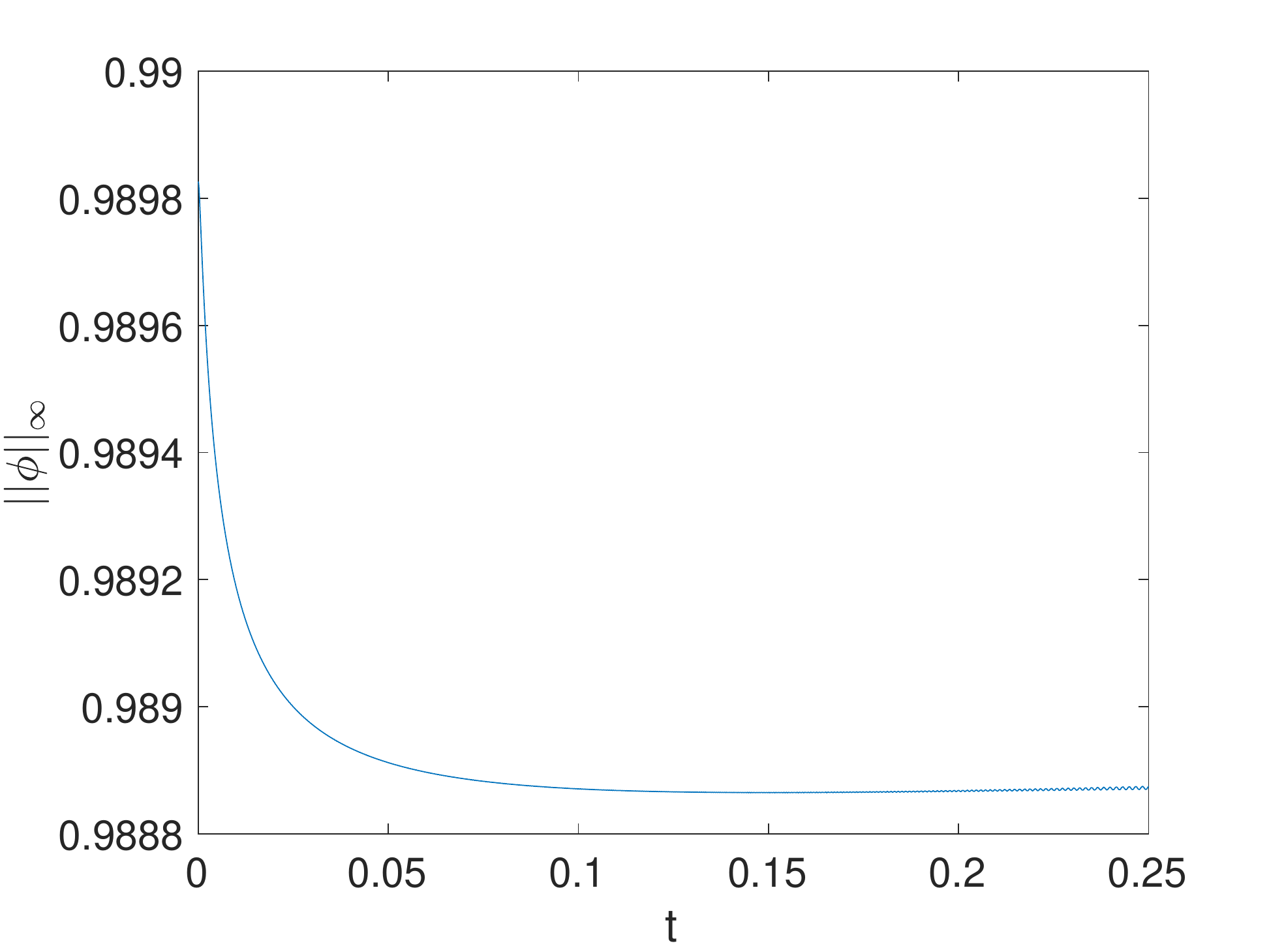}
  \caption{$L^{\infty}$ norms of the solution of (\ref{eqmodtime}) 
  for the initial data $\phi(x,0)=\varphi(x)\pm0.001\exp(-x^{2})$, on the 
  left for the minus sign, on the right for the plus sign in the 
  initial data.  }
  \label{NLSGSa9b44m001gaussinf}
\end{figure}

In Fig.~\ref{NLSGSa9b44m001gaussfit} we show the solutions for both 
cases at the 
final time in blue together with fitted ground states. In both cases 
the agreement is so good that a difference (again of the order of 
$10^{-3}$) can hardly be recognized. Thus the ground
states appear to be asymptotically stable also in this case. The 
fitting yields  $b=4.3992$ for the $-$ sign and $b=4.4008$ for the 
$+$ sign, i.e., the expected values close to $4.4$. 
\begin{figure}[htb!]
  \includegraphics[width=0.45\textwidth]{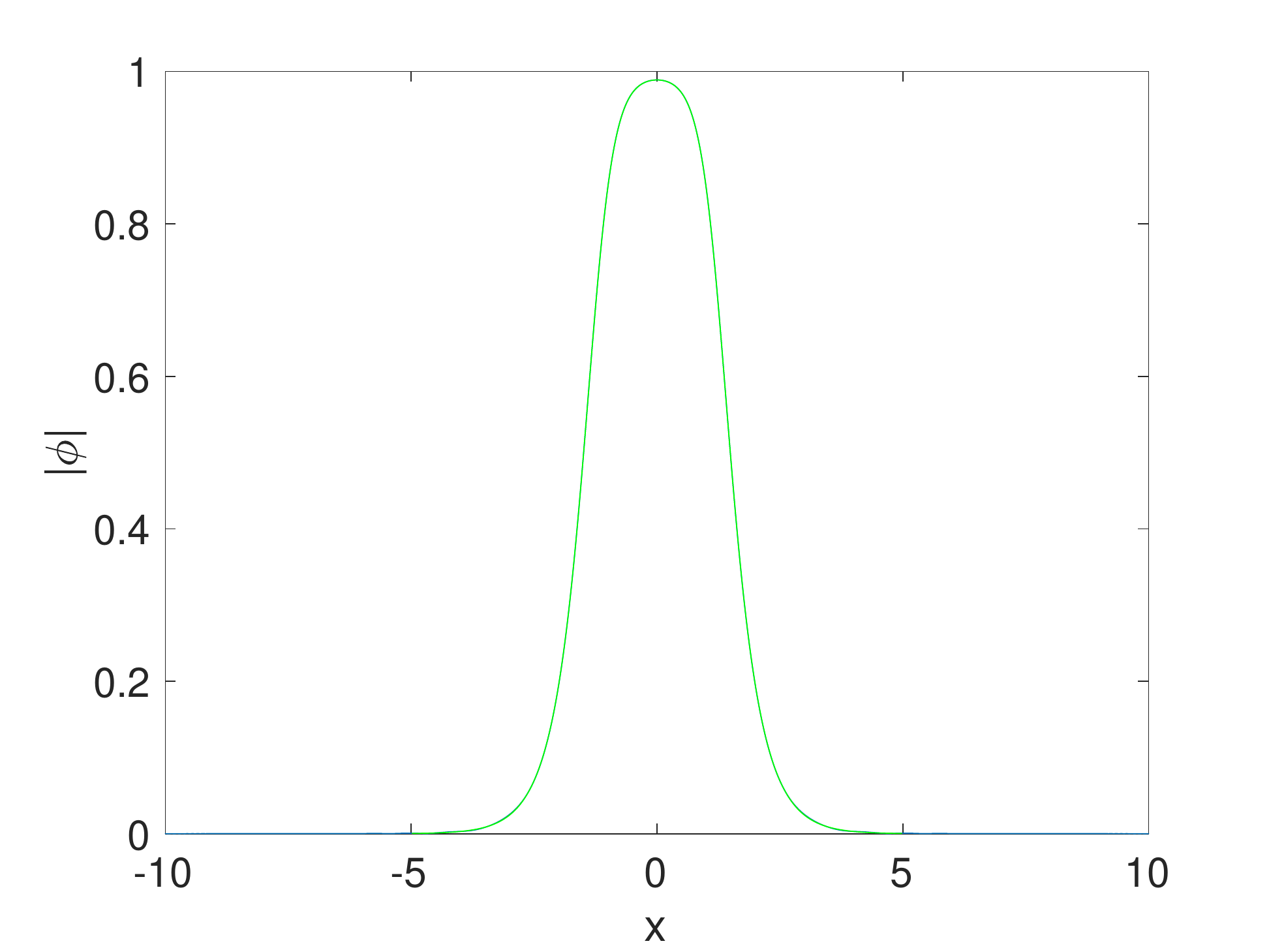}
  \includegraphics[width=0.45\textwidth]{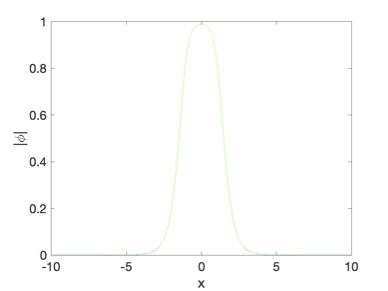}
  \caption{Solutions of (\ref{eqmodtime}) 
  for the initial data $\phi(x,0)=\varphi(x)\pm0.001\exp(-x^{2})$, on the 
  left for the minus sign, on the right for the plus sign in the 
  initial data in blue together with a fitted ground state in green.  }
  \label{NLSGSa9b44m001gaussfit}
\end{figure}

\paragraph{\textbf{Higher nonlinearities}}

We repeat the experiments of Fig.~\ref{figGS9.44.099} and 
Fig.~\ref{figGS9.44.1001} for $\alpha=2$, i.e., a higher 
nonlinearity. As can be seen below, the ground states still appear to 
be stable, but take a considerably longer time to settle to a final 
state. This means we will need much higher numerical resolution in 
order to avoid too much interaction between the radiation and the 
bulk on a torus (we simply choose a larger period), and have to solve 
for longer times. 
We use $N=2^{13}$ Fourier modes for $x\in20[-\pi,\pi]$ and 
$N_{t}=2*10^{6}$ time steps for $t\in[0,2]$. 

The $L^{\infty}$ norms for the perturbed ground state 
can be seen in Fig.~\ref{NLSGSa9b2alpha29_099inf} on the left. Again a ground 
state with slightly 
smaller maximum than the perturbed gound state appears to be reached 
for $\lambda=0.99$. But 
this time the $L^{\infty}$ norm performs some damped oscillations 
around what seems to be an 
asymptotic value. Since there is no dissipation in the system, this 
settling into a final state can take quite long compared to a period 
of the initial ground state, and interestingly takes much longer than in 
the case $\alpha=1$ above\footnote{The very small oscillations which 
appear for larger times on the $L^{\infty}$ norm are due to us studying the perturbations in a 
periodic setting and not on $\mathbb{R}$. Thus radiation emitted to 
infinity appears on the other side of the computational domain and 
interacts after some time with the bulk which leads to small
periodic excitations of the latter. This effect can be fully 
suppressed by considering larger periods.}. Note that the quintic nonlinearity 
is $L^{2}$ critical in the standard NLS equation, i.e., solutions to 
initial data of sufficient mass blow up in finite time. Here no 
blow-up is observed, $||\phi||_{\infty}<1$ for all times.  The oscillations around some finite value for the 
$L^{\infty}$ norm can also be seen on the right of 
Fig.~\ref{NLSGSa9b2alpha29_099inf}, where the solution at the final 
time is shown together with a ground state of the fitted 
asymptotic maximum ($b\sim 2.9221$). It can be seen that the fitting 
is not as good  as in the case 
$\alpha=1$ shown above, and that the found value for $b$ is slightly 
larger than the original one. This means that the solution at 
the final time is not yet sufficiently close to the asymptotic 
solution. 
\begin{figure}[htb!]
  \includegraphics[width=0.45\textwidth]{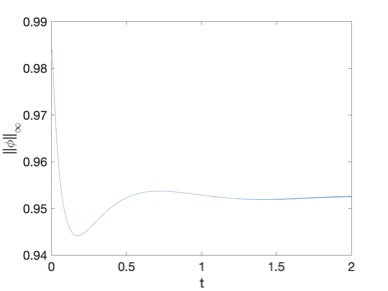}
  \includegraphics[width=0.45\textwidth]{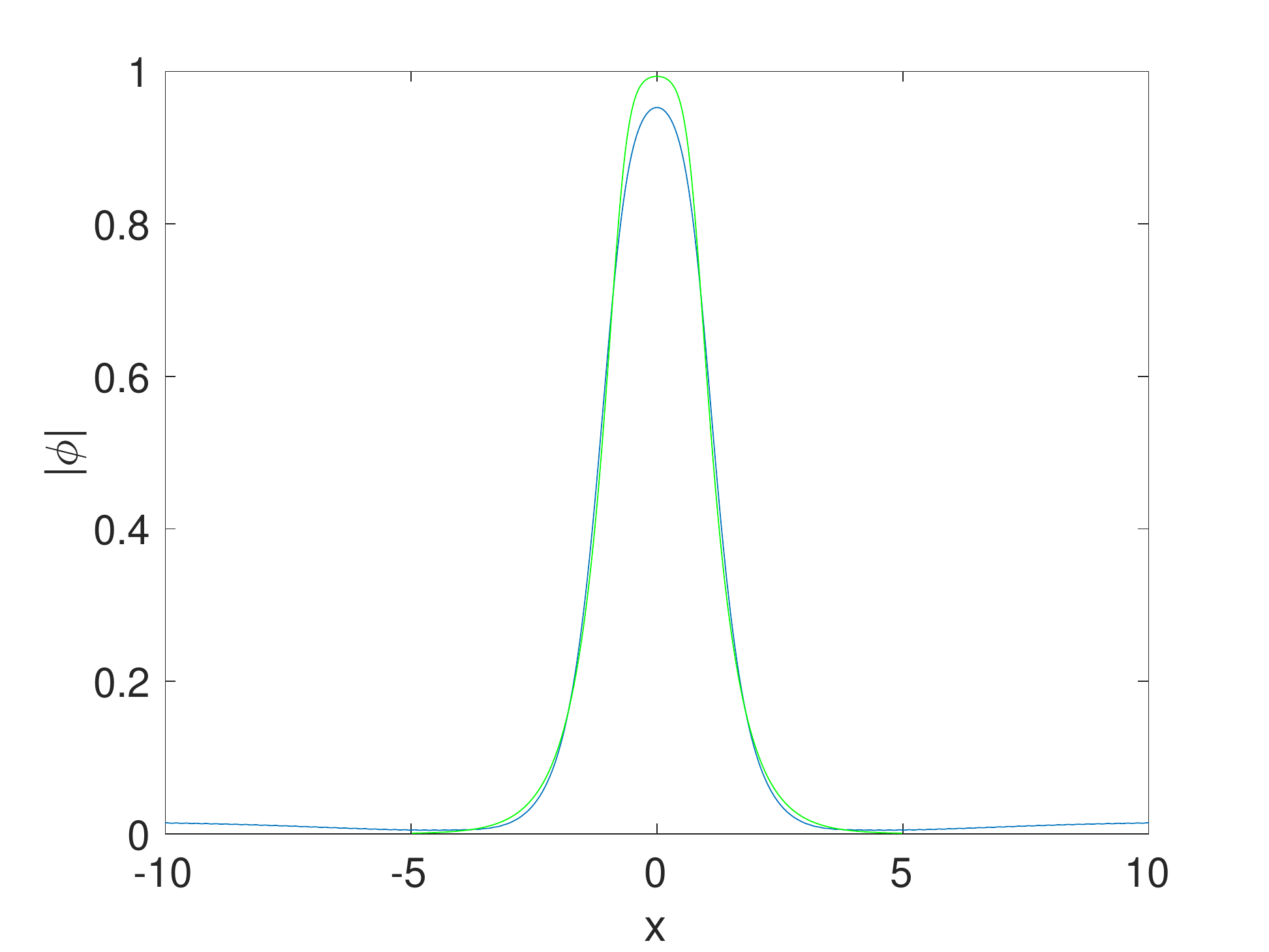}
  \caption{$L^{\infty}$ norm of the solution of (\ref{eqmodtimealpha}) 
  for the initial data $\phi(x,0)=\lambda\varphi(x)$ and $a=9$, 
  $b=2.9$ and $\alpha=2$ for $\lambda=0.99$ on the left,  on the 
  right the solution at the final time in blue together with an 
  estimated ground state in green.  }
  \label{NLSGSa9b2alpha29_099inf}
\end{figure}

The situation is very similar for a perturbation with $\lambda=1.001$ 
shown in Fig.~\ref{NLSGSa9b2alpha29_1001inf}. The $L^{\infty}$ on the 
left of Fig.~\ref{NLSGSa9b2alpha29_1001inf} again 
appears to oscillate around some asymptotic value which is not fully 
reached during our computation. A fitted value ($b\sim 2.9315$ for this final 
state leads to the green curve on the right of the same figure. It 
shows that the final state is not yet reached, but close to the green 
curve. Note that the ground state appears to be stable also for the Gaussian 
perturbations of Fig.~\ref{NLSGSa9b44m001gaussinf} which are not 
shown here.
\begin{figure}[htb!]
  \includegraphics[width=0.45\textwidth]{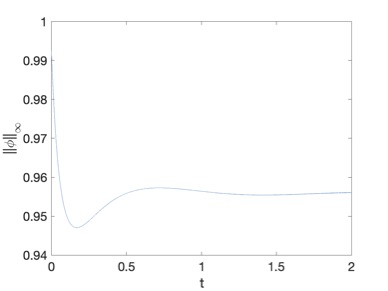}
  \includegraphics[width=0.45\textwidth]{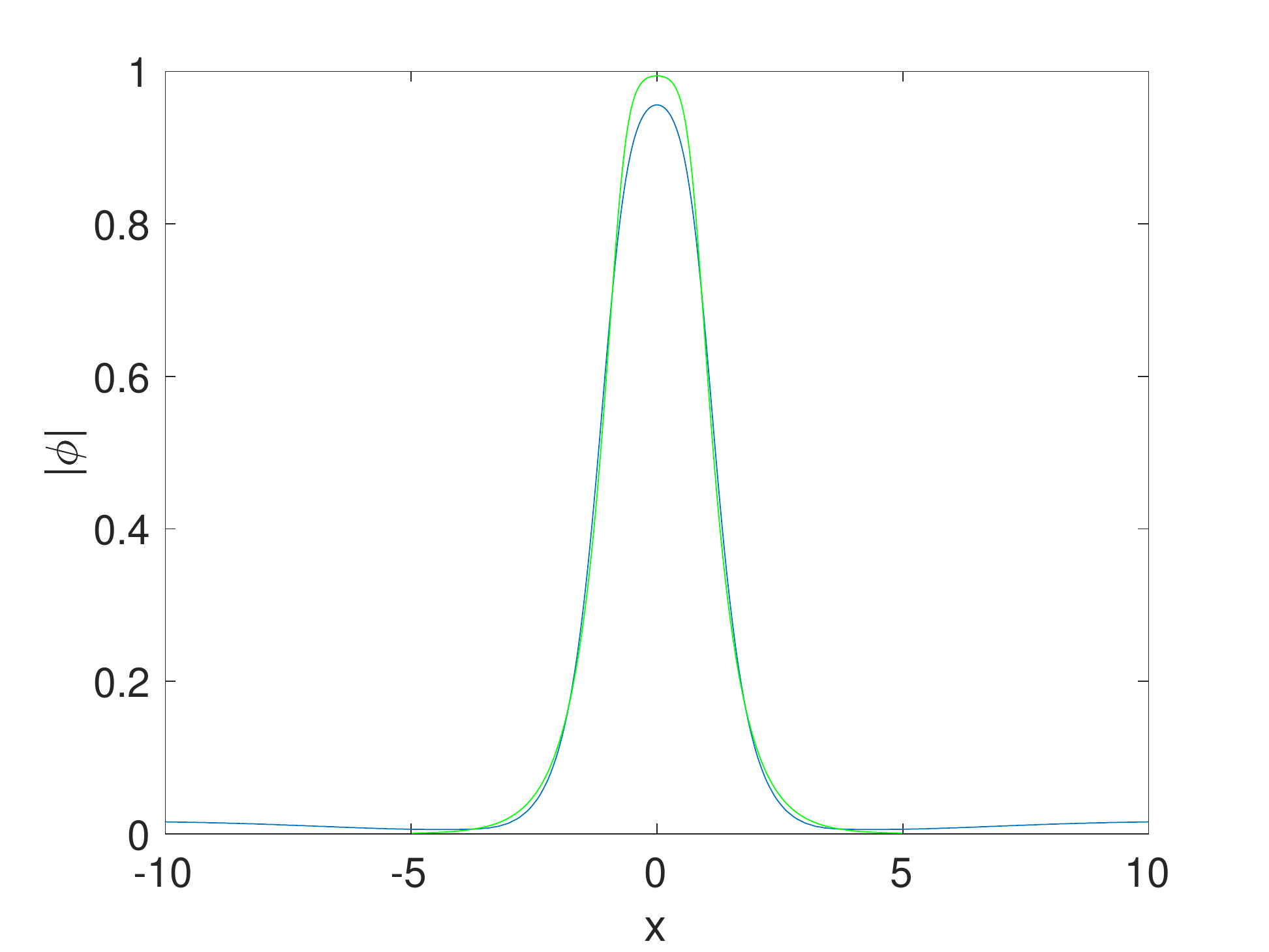}
  \caption{$L^{\infty}$ norm of the solution of (\ref{eqmodtimealpha}) 
  for the initial data $\phi(x,0)=\lambda\varphi(x)$ and $a=9$, 
  $b=2.9$ and $\alpha=2$ for $\lambda=1.001$ on the left,  on the 
  right the solution at the final time in blue together with an 
  estimated ground state in green.  }
  \label{NLSGSa9b2alpha29_1001inf}
\end{figure}

The same experiments as above are shown for an even higher 
nonlinearity $\alpha=3$  and $b=2.1$ in Fig.~\ref{NLSGSa9b21alpha3_099inf} for 
$\lambda=0.99$. The $L^{\infty}$ norm on the left of the figure 
appears to oscillate around some some asymptotic value. On the right 
of the same figure we show the solution at the final time in blue 
plus a fitted ($b\sim2.1492$) ground state in green. Once 
more the ground states appear to be stable (also for Gaussian perturbations 
not shown here), but the final state will be only fully reached at 
longer times (the fitted value of $b$ is even larger than the 
original here). 
\begin{figure}[htb!]
  \includegraphics[width=0.45\textwidth]{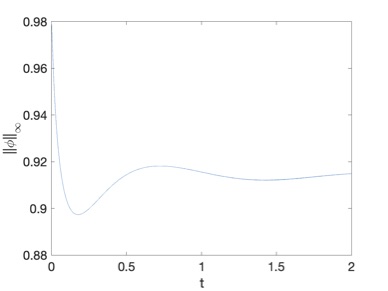}
  \includegraphics[width=0.45\textwidth]{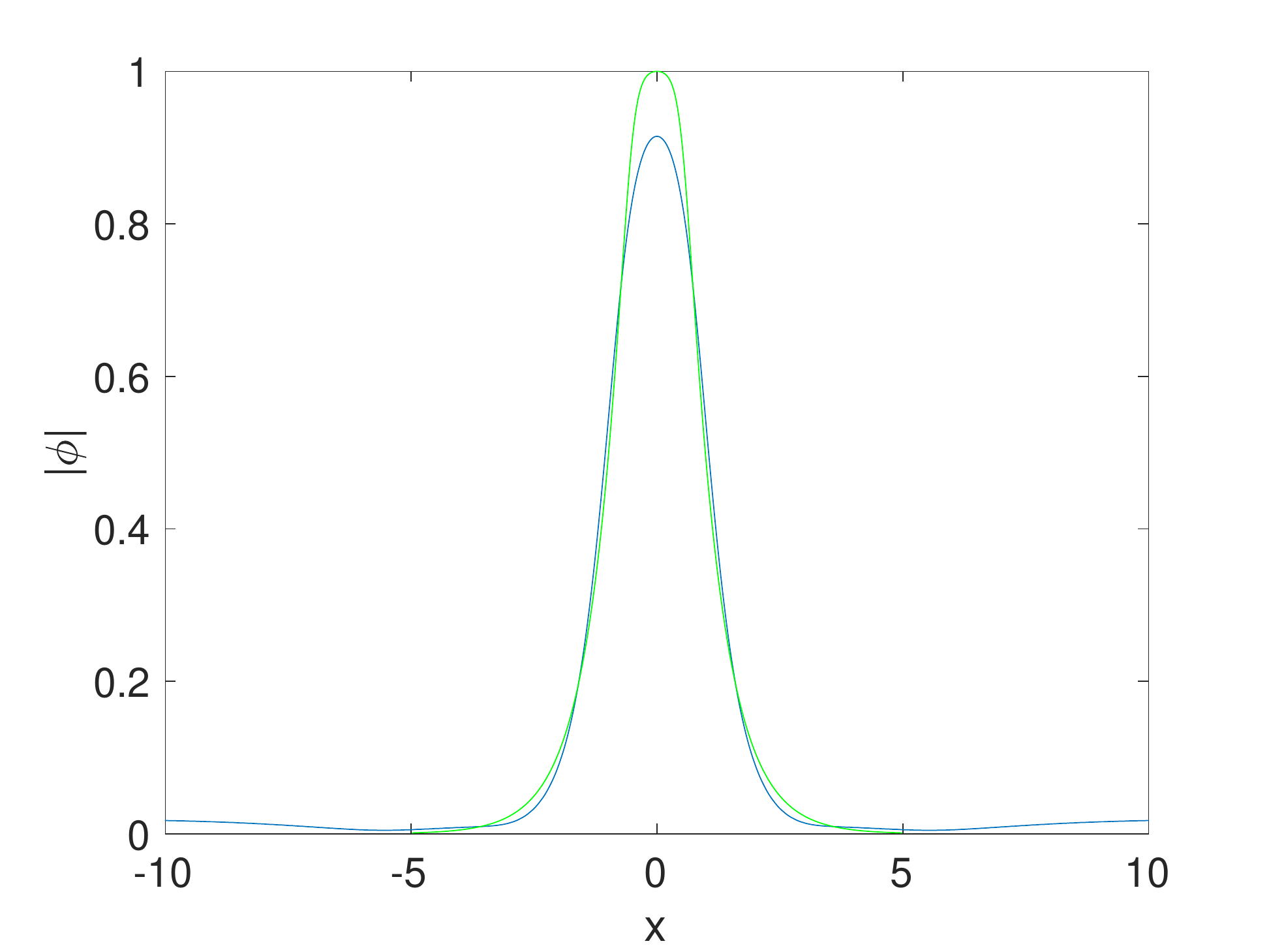}
  \caption{$L^{\infty}$ norm of the solution of (\ref{eqmodtimealpha}) 
  for the initial data $\phi(x,0)=0.99\varphi(x)$ and $a=9$, 
  $b=2.1$ and $\alpha=3$ on the left,  on the 
  right the solution at the final time in blue together with an 
  estimated ground state in green.  }
  \label{NLSGSa9b21alpha3_099inf}
\end{figure}

The situation is similar for $\lambda=1.001$ shown in 
Fig.~\ref{NLSGSa9b21alpha3_1001inf}. On the left the $L^{\infty}$ 
norm appears to oscillate around some asymptotic value. On the right 
we show the solution at the final time together with a fitted  
($b\sim2.164$) ground state solution in green. 
Note that for the standard NLS a septic nonlinearity 
would be $L^{2}$ supercritical which again would lead to a blow-up of initial 
data of sufficiently large mass.
\begin{figure}[htb!]
  \includegraphics[width=0.45\textwidth]{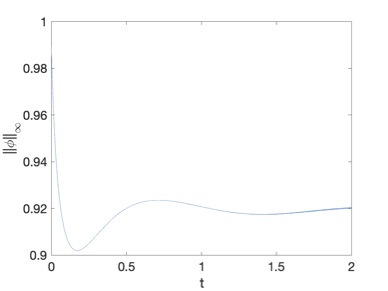}
  \includegraphics[width=0.45\textwidth]{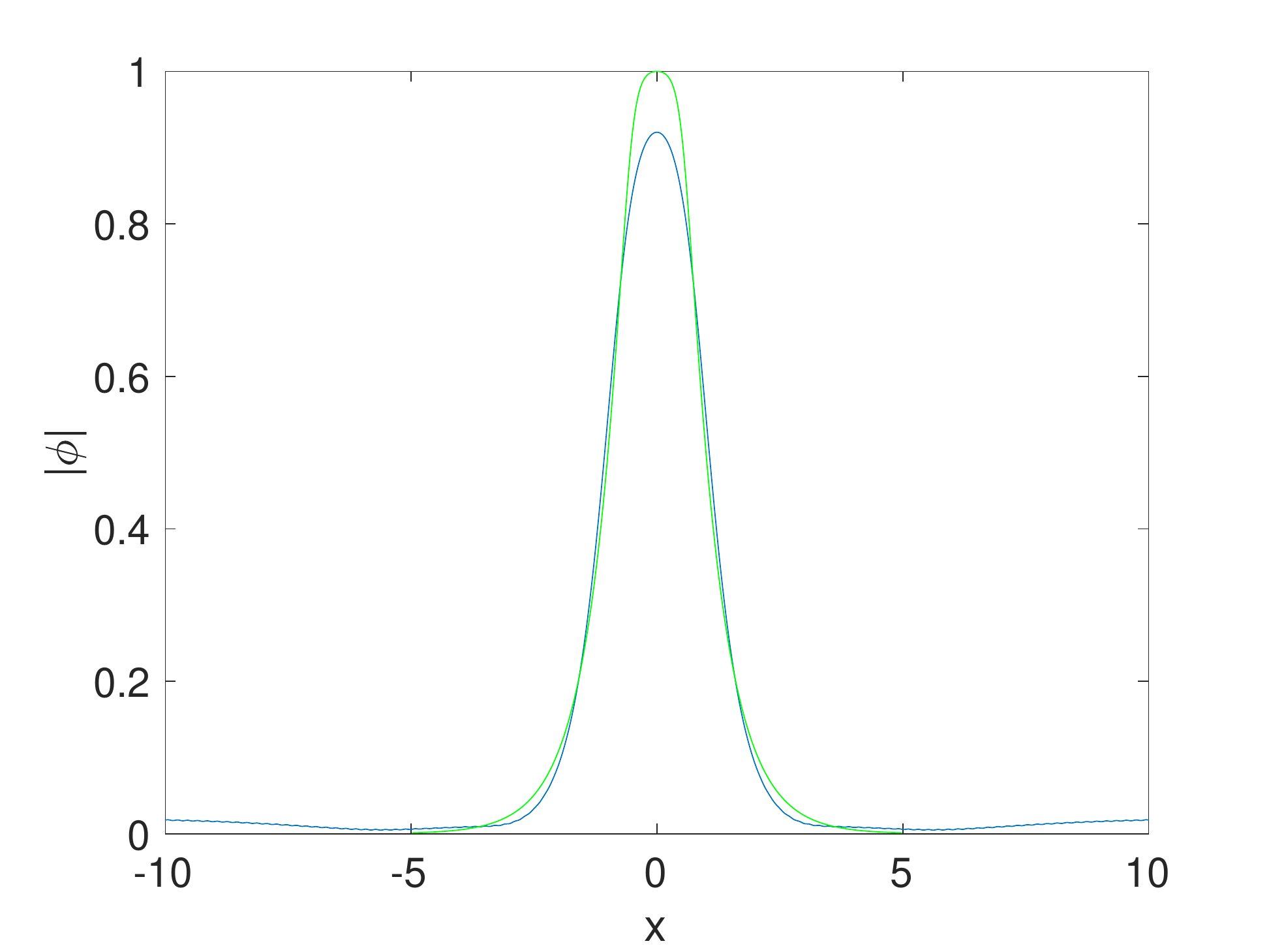}
  \caption{$L^{\infty}$ norm of the solution of (\ref{eqmodtimealpha}) 
  for the initial data $\phi(x,0)=1.001\varphi(x)$ and $a=9$, 
  $b=2.1$ and $\alpha=3$  on the left,  on the 
  right the solution at the final time in blue together with an 
  estimated ground state in green.  }
  \label{NLSGSa9b21alpha3_1001inf}
\end{figure}

\subsection{Schwartz class initial data}

An interesting question in this context is whether 
these stable ground states appear in the long time behavior of 
solutions to generic localized initial data. To address this 
question we consider initial data of the form $\phi(x,0)=\mu 
\exp(-x^{2})$ with $0<\mu<1$, again for $a=9$. We use $N=2^{12}$ 
Fourier modes for $x\in 40[-\pi,\pi]$ and $N_{t}=5*10^{5}$ time steps 
for the indicated time intervals. In Fig.~\ref{fig09gauss} it can be seen that the 
$L^{\infty}$ norm of the solution appears to oscillate around some 
asymptotic values, and that some radiation is emitted towards 
infinity. 
\begin{figure}[htb!]
  \includegraphics[width=0.7\textwidth]{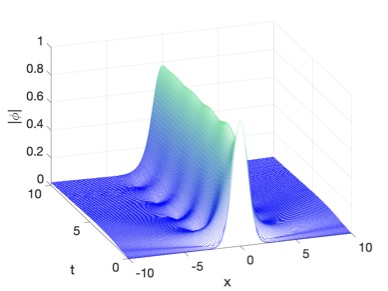}
  \caption{Solution to the equation (\ref{eqmodtime}) with $a=9$
  for the initial 
  data $\phi(x,0)=0.9\exp(-x^{2})$. }
 \label{fig09gauss}
\end{figure}

The former effect is more visible on the left of 
Fig.~\ref{fig09gaussinf} where the $L^{\infty}$ norm of the solution 
is shown. Since there is no dissipation in the system, the final 
ground state will be only reached asymptotically. On the right of the 
same figure we show the solution at the final time of the computation 
in blue together with an estimated ground state ($b\sim2.7188$) in 
green. 
\begin{figure}[htb!]
  \includegraphics[width=0.45\textwidth]{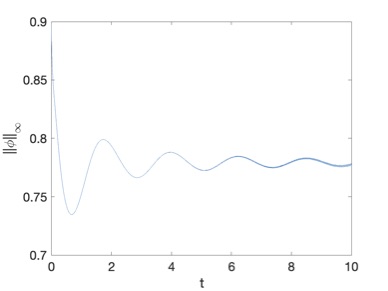}
  \includegraphics[width=0.45\textwidth]{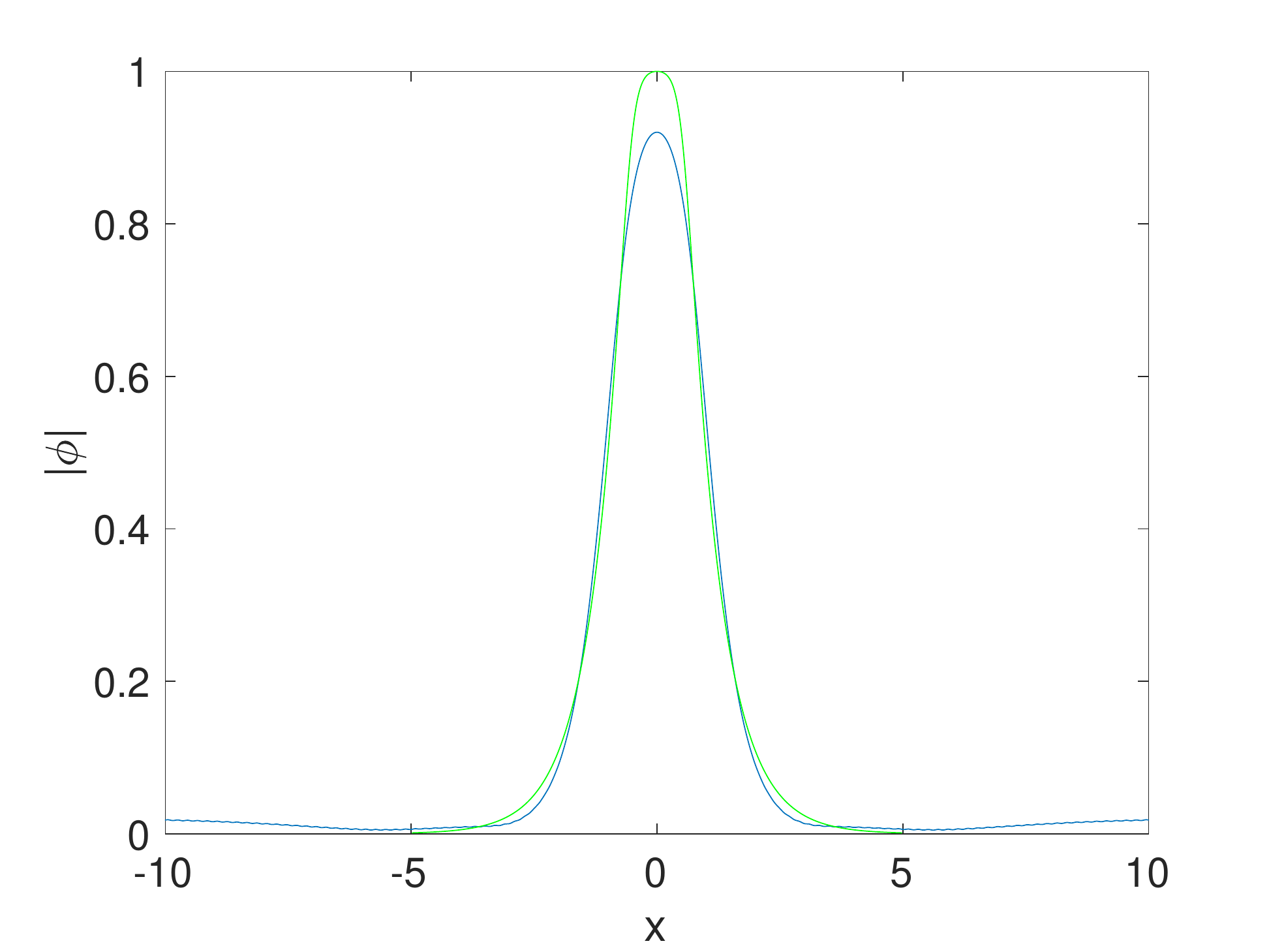}
  \caption{On the left the $L^{\infty}$ norm of the solution of 
  Fig.~\ref{fig09gauss}, on the right the  solution at the final time 
  in blue together with a fitted ground state in green.  }
 \label{fig09gaussinf}
\end{figure}

The situation is similar for higher nonlinearity. In 
Fig.~\ref{NLSGSa9alpha2_09gaussinf} we show the case $\alpha=2$. On 
the left for $\mu=0.9$, the $L^{\infty}$ norm of the solution appears 
again to show damped oscillations around some asymptotic value, 
presumably a ground state. The solution at the final computed time is 
shown on the right of the same figure together with a fitted 
($b=1.4399$) ground state in green. Though the final state is not yet 
reached, it appears that the soliton resolution conjecture also 
applies to this case. 
\begin{figure}[htb!]
  \includegraphics[width=0.45\textwidth]{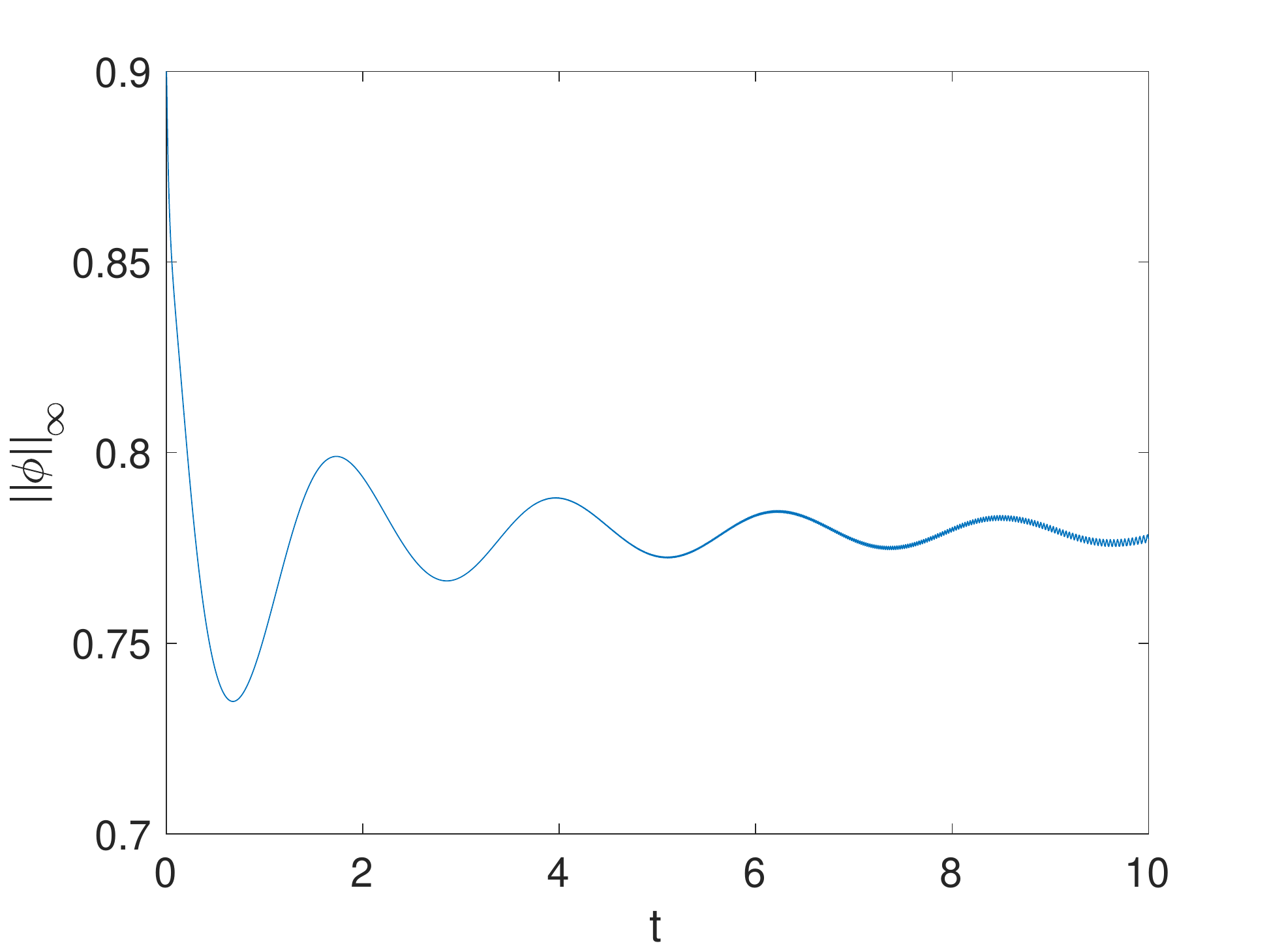}
  \includegraphics[width=0.45\textwidth]{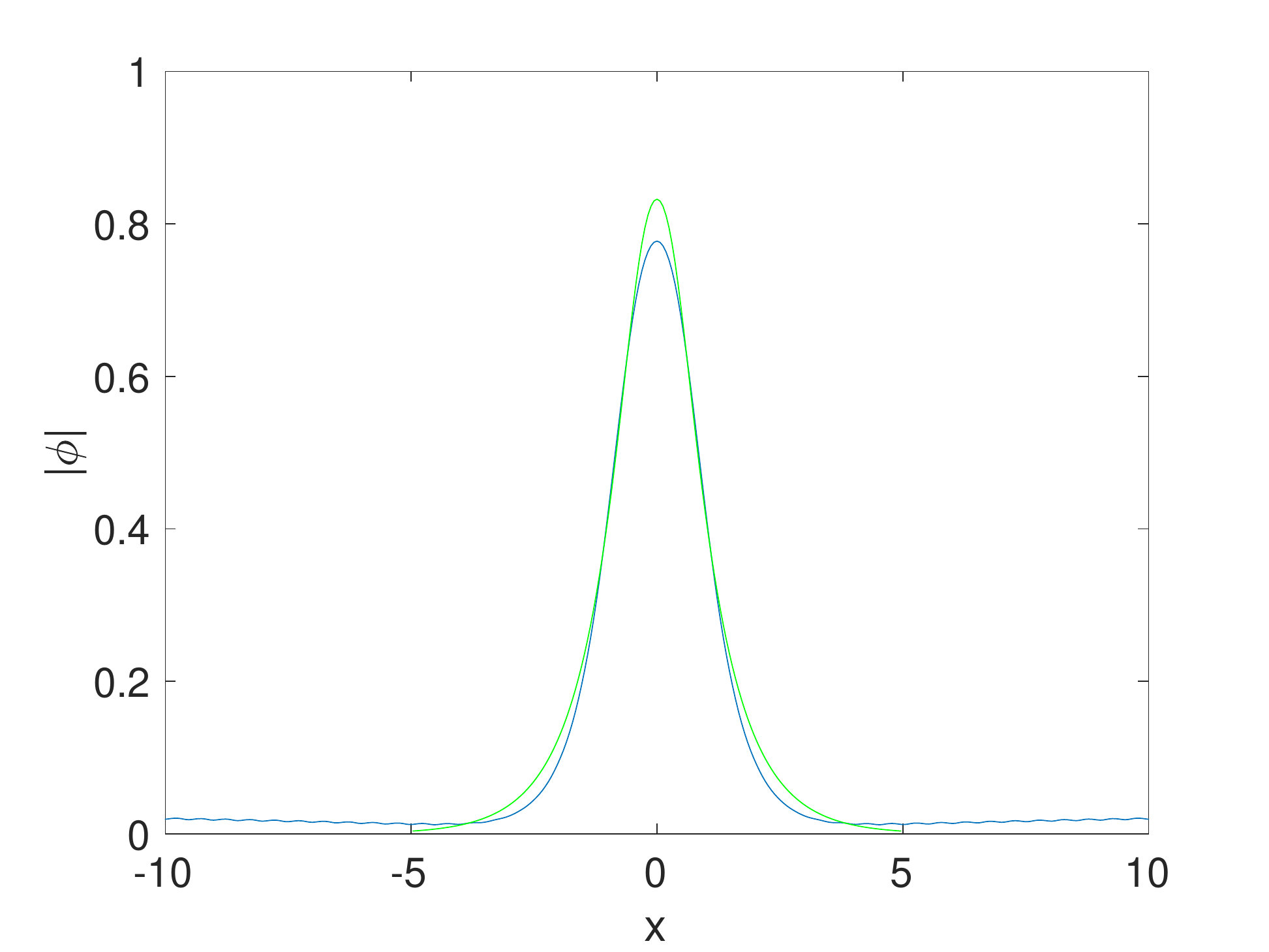}
  \caption{On the left the $L^{\infty}$ norm of the solution of 
  equation (\ref{eqmodtimealpha}) for $\alpha=2$ and the initial data 
  $\phi(x,0)=0.9\exp(-x^{2})$, and the solution for $t=10$ on the 
  right in blue together with some fitted ground state in green.}
 \label{NLSGSa9alpha2_09gaussinf}
\end{figure}

For the even higher nonlinearity $\alpha=3$,  we consider in 
Fig.~\ref{NLSGSa9alpha3_09gaussinf} the case $\mu=0.9$ and show on 
the left the $L^{\infty}$ norm of the solution which once more shows 
damped oscillations around some final state. On the right of the same 
figure we give the solution at the final computed time together with 
a fitted ground state ($b=1.2549$). Once more the fitting is not 
perfect since the final state of the solution is not yet reached, but 
it appears plausible that this final state is indeed a ground state.  
\begin{figure}[htb!]
  \includegraphics[width=0.45\textwidth]{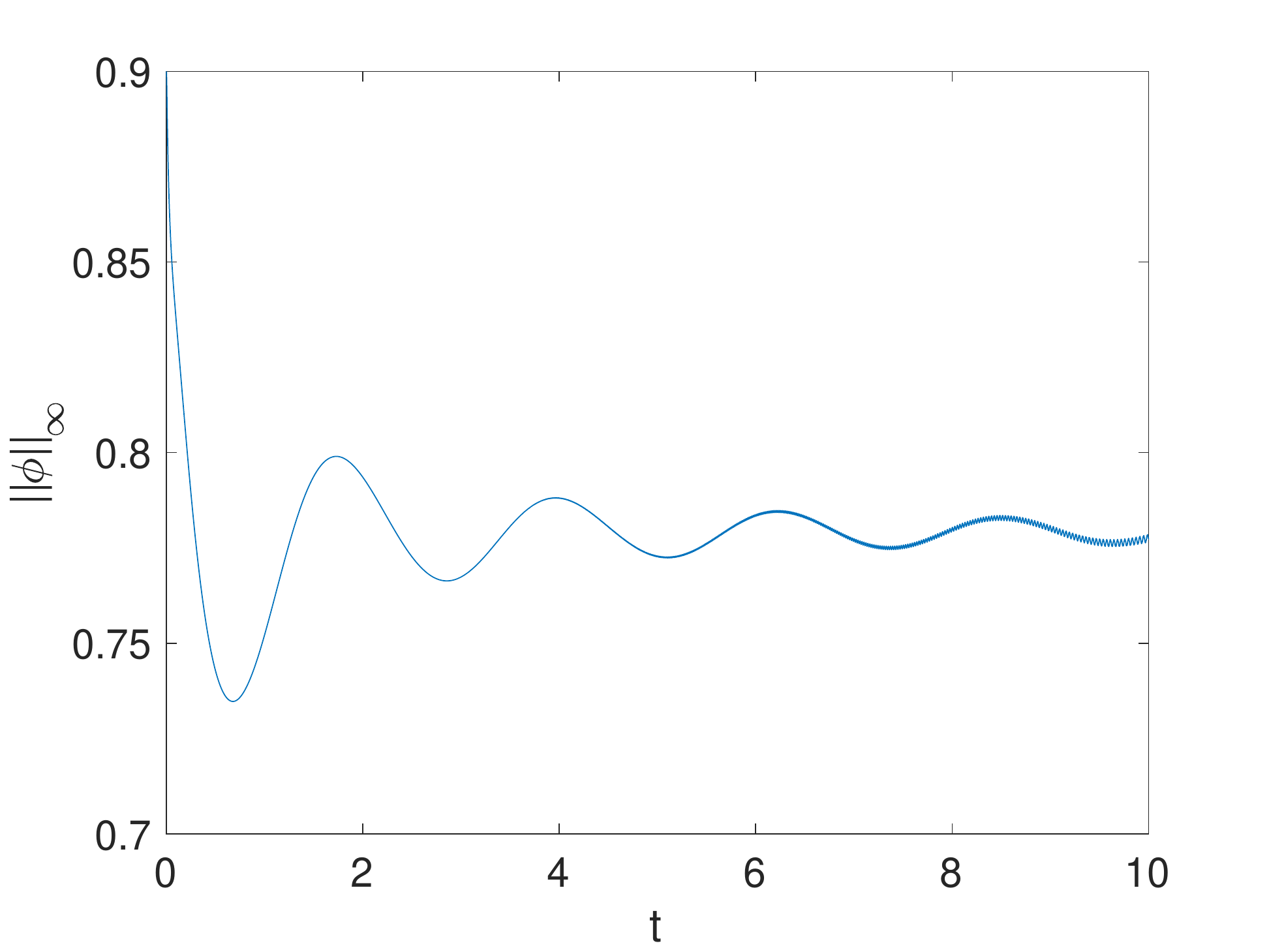}
  \includegraphics[width=0.45\textwidth]{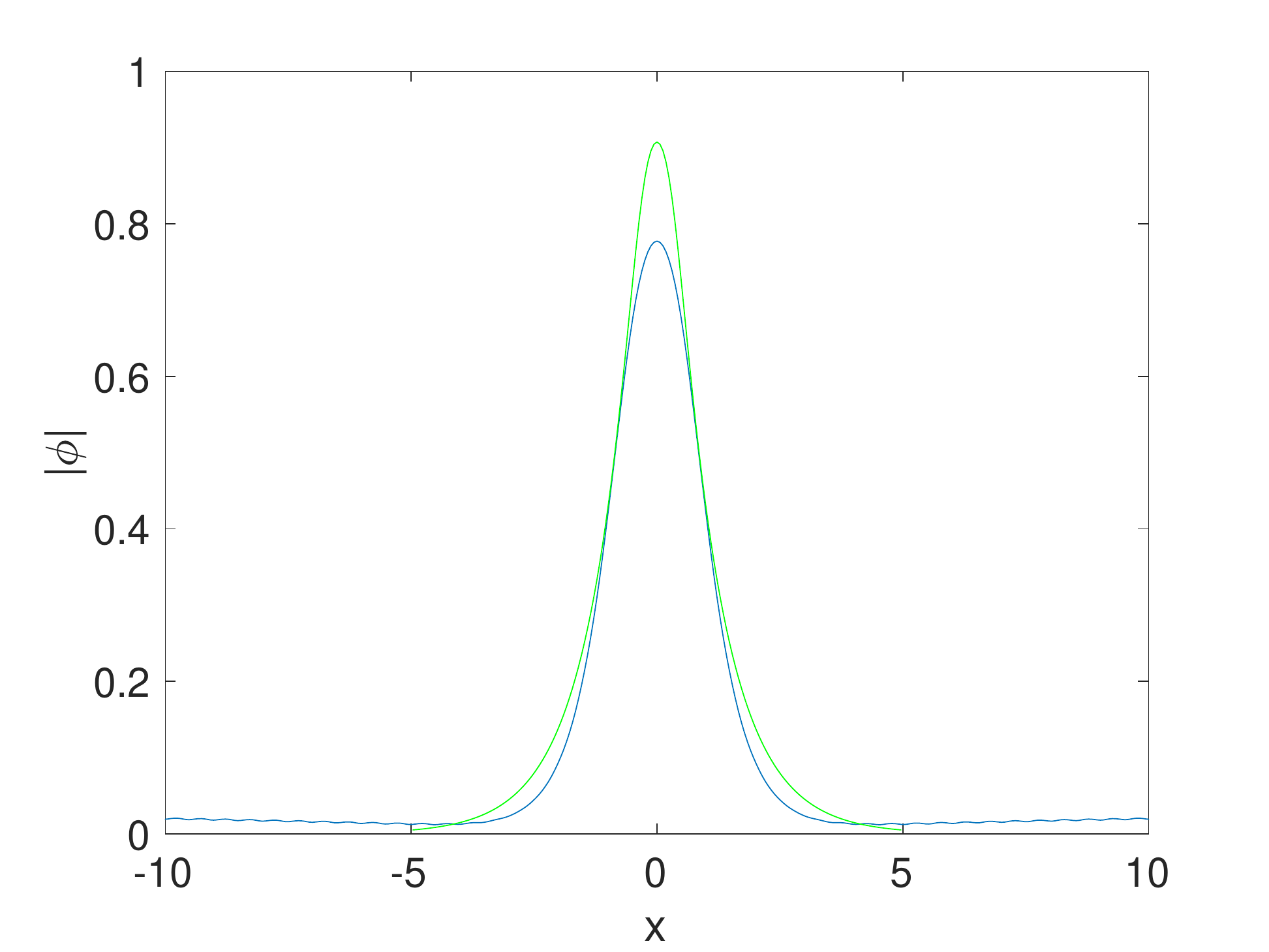}
  \caption{On the left the $L^{\infty}$ norm of the solution of 
  equation (\ref{eqmodtimealpha}) for $\alpha=3$ and the initial data 
  $\phi(x,0)=0.9\exp(-x^{2})$, on the right the solution for $t=10$ 
  in blue  together with some fitted ground state.  }
 \label{NLSGSa9alpha3_09gaussinf}
\end{figure}

\section{Outlook: analysis of the model in higher space dimension}\label{secoutlook}

It seems interesting to investigate the model described in this paper in higher space dimension, $d=3$ being the most relevant case from a physical point of view.

On the one hand, to generalize the model for any dimension $d>1$, one can simply replace $\partial_x$ in equation~\eqref{eqmodtimealpha} by the operator $\nabla$. This leads to a quasilinear Schrödinger equation of the form 
\begin{equation*}
i\partial_t\phi=-\nabla\cdot\left(\frac{\nabla\phi}{1-|\phi|^{2\alpha}}\right)+\alpha|\phi|^{2\alpha-2}\frac{|\nabla\phi|^2}{(1-|\phi|^{2\alpha})^2}\phi-a|\phi|^{2\alpha}\phi
\end{equation*} 
with $\phi\in L^2(\R^d,\C)$.
On the other hand, at least in dimension $d=2$ and $d=3$, another possibility is to formally derive the equation of the model by following the arguments presented in Appendix~\ref{formalderivation} and by taking as starting point the Dirac equation in dimension $1<d\le 3$. This will lead to a slightly more complicated quasilinear Schrödinger equation.            

In both cases, solitary wave solutions for any $\alpha\in \N^*$ are 
expected to exist, and one can investigate their behavior 
analytically and numerically. However, the study of the 
time-dependent equation from an analytical point of view seems much more involved.

From a numerical point of view, a stiff time integrator would be 
recommended for higher dimensions in order to overcome stability 
constraints, in particular if dispersive shock waves are studied 
instead of perturbations of ground states as in the present paper. 
A straightforward, but computationally expensive 
approach would be an implicit scheme, which probably will have to be 
combined with some Newton-Krylov scheme since a standard fixed point 
iteration might not converge except for very small time steps. More interesting would be 
Rosenbrock-type integrators based on a linearization of the equation 
after the spatial discretisation, and an exponential integrator for 
the Jacobian of the resulting system. This can be done efficiently by using 
so-called Leja points, see for instance \cite{leja} and references 
therein.

We leave this generalization to future work.

\begin{appendix}

\section{Formal derivation of the non-relativistic model}\label{formalderivation}

Consider the following nonlinear Dirac equation in one space dimension

\begin{equation}\label{TNLDeq}
i\partial_t \Psi=(-i\sigma_1\partial_x+\sigma_3 m)\Psi-\kappa_1(\langle \sigma_3 \Psi,\Psi\rangle)^\alpha\sigma_3\Psi +\kappa_2 |\Psi|^{2\alpha}\Psi
\end{equation}

with $\kappa_1$ and $\kappa_2$ positive constants and $\alpha\in \NN^*$. Here $\Psi=(\psi,\zeta)$ is a 2-spinor that describes the quantum state of a nucleon of mass $m$, and $\sigma_1$ and $\sigma_3$ are the Pauli matrices given by
\begin{equation*}
\sigma_1=\begin{pmatrix}0 & 1\\ 1 &0\end{pmatrix},
\ \sigma_3=\begin{pmatrix}1 & 0\\ 0 &-1\end{pmatrix} .
\end{equation*}

In nuclear physics, the interesting regime is when the parameters $\kappa_1$ and $\kappa_2$ behave like $m$, whereas $\kappa_1-\kappa_2$ stays bounded. More precisely, let $\kappa_1=\theta m$
and $\kappa_1-\kappa_2=\lambda$ with $\theta$ and $\lambda$ positive constants. As a consequence, the nonlinear Dirac equation~\eqref{TNLDeq} can be written as
\begin{equation}\label{TNLDsys}
\left\{
\begin{aligned}
&i\partial_t\psi=-i\partial_x \zeta+m\psi-\theta m 
(|\psi|^2-|\zeta|^2)^\alpha\psi+(\theta m-\lambda)(|\psi|^2+|\zeta|^2)^\alpha\psi,\\
&i\partial_t\zeta=-i\partial_x\psi-m\zeta+\theta m 
(|\psi|^2-|\zeta|^2)^\alpha\zeta+(\theta m-\lambda)(|\psi|^2+|\zeta|^2)^\alpha\zeta.
\end{aligned}
\right.
\end{equation}
Hence, by writing
$$
\tilde \phi(t,x)=e^{imt}\psi(t,x)\text{ and } \tilde \chi(t,x)=e^{imt}\zeta(t,x),
$$
we obtain 
\begin{equation*}
\left\{
\begin{aligned}
&i\partial_t\tilde \phi=-i\partial_x \tilde \chi+\theta m 
\left(\sum_{k=0}^\alpha{\alpha \choose k}(|\tilde\phi|^2)^{\alpha-k}|\tilde\chi|^{2k}((-1)^{k+1}+1) \right)\tilde\phi-\lambda \left(\sum_{k=0}^\alpha{\alpha \choose k}(|\tilde\phi|^2)^{\alpha-k}|\tilde\chi|^{2k} \right)\tilde\phi,\\
&i\partial_t\tilde \chi=-i\partial_x \tilde\phi-2m\tilde\chi+\theta m 
\left(\sum_{k=0}^\alpha{\alpha \choose k}(|\tilde\phi|^2)^{\alpha-k}|\tilde\chi|^{2k}((-1)^{k}+1) \right)\tilde \chi-\lambda\left(\sum_{k=0}^\alpha{\alpha \choose k}(|\tilde\phi|^2)^{\alpha-k}|\tilde\chi|^{2k} \right)\tilde\chi.
\end{aligned}
\right.
\end{equation*}
As usual, in the non-relativistic regime, the lower spinor $\tilde \chi$ is of order $1/\sqrt{m}$. Hence, we have to perform the following change of scale
$$
\tilde \phi(t,x)=\left(\frac{1}{\theta}\right)^{\frac{1}{2\alpha}}\phi\left(\frac{t}{2},\sqrt{m}x\right)\text{ and } \tilde \chi(t,x)=\left(\frac{1}{\theta}\right)^{\frac{1}{2\alpha}}\frac{1}{2\sqrt{m}}\chi\left(\frac{t}{2},\sqrt{m}x\right)
$$
which leads to 
\begin{equation*}
\left\{
\begin{aligned}
&i\partial_t\phi=-i\partial_x \chi+\alpha(|\phi|^2)^{\alpha-1}|\chi|^2\phi-a|\phi|^{2\alpha}\phi+\frac{1}{m}F(1/m,\phi,\chi)\phi\\
&\frac{1}{m}\frac{1}{4}i\partial_t\chi=-i\partial_x\phi-\chi+|\phi|^{2\alpha}\chi-\frac{1}{m}G(1/m,\phi,\chi)\chi
\end{aligned}
\right.
\end{equation*}
with $a=\frac{2\lambda}{\theta}$, and $F, G$ defined by 
\begin{align*}
F(1/m,\phi,\chi)&=2\sum_{k=2}^\alpha{\alpha \choose k}(|\phi|^2)^{\alpha-k}\frac{|\chi|^{2k}}{2^{2k}m^{k-2}}((-1)^{k+1}+1)-a\sum_{k=1}^\alpha{\alpha \choose k}(|\phi|^2)^{\alpha-k}\frac{|\chi|^{2k}}{2^{2k}m^{k-1}},\\
G(1/m,\phi,\chi)&=\frac{1}{2}\sum_{k=1}^\alpha{\alpha \choose k}(|\phi|^2)^{\alpha-k}\frac{|\chi|^{2k}}{2^{2k}m^{k-1}}((-1)^{k}+1)-\frac{a}{4}\sum_{k=0}^\alpha{\alpha \choose k}(|\phi|^2)^{\alpha-k}\frac{|\chi|^{2k}}{2^{2k}m^{k}}.
\end{align*}

Finally, denoting $\varepsilon=\frac{1}{m}$ the perturbative parameter, we obtain 
\begin{equation}\label{TNLDsyseps}
\left\{
\begin{aligned}
&i\partial_t\phi=-i\partial_x 
\chi+\alpha(|\phi|^2)^{\alpha-1}|\chi|^2\phi-a|\phi|^{2\alpha}\phi+\varepsilon F(\varepsilon,\phi,\chi)\phi,\\
&i\partial_x\phi+(1-|\phi|^{2\alpha})\chi+\varepsilon\frac{1}{4}i\partial_t\chi+\varepsilon G(\varepsilon,\phi,\chi)\chi=0.
\end{aligned}
\right.
\end{equation}

In particular, when $\varepsilon=0$, we have 
\begin{equation}\label{TNLDsysnonrel}
\left\{
\begin{aligned}
&i\partial_t\phi=-i\partial_x 
\chi+\alpha(|\phi|^2)^{\alpha-1}|\chi|^2\phi-a|\phi|^{2\alpha}\phi,\\
&i\partial_x\phi+(1-|\phi|^{2\alpha})\chi=0,
\end{aligned}
\right.
\end{equation}
which leads at least formally to the time-dependent quasilinear Schrödinger equation 
\begin{equation*}
i\partial_t\phi=-\partial_x\left(\frac{\partial_x\phi}{1-|\phi|^{2\alpha}}\right)+\alpha(|\phi|^2)^{\alpha-1}\frac{|\partial_x\phi|^2}{(1-|\phi|^{2\alpha})^2}\phi-a|\phi|^{2\alpha}\phi.
\end{equation*} 

\section{Existence of positive solutions to the stationary equation~\eqref{eqmodstat}}\label{appexistencesolstat}

In this appendix, we prove Theorem~\ref{existence} and we make rigorous the construction of ground states presented in Section~\ref{secgroundstates}.

As in~\cite{EstRot-12}, we write the stationary equation~\eqref{eqmodstat} as a system of first order ODEs

\begin{equation}\label{eqstatsys}
\left\{
\begin{aligned}
&\chi'=\alpha \chi^2\varphi^{2\alpha-1}-a\varphi^{2\alpha+1}+b\varphi\\
&\varphi'= \chi(1-\varphi^{2\alpha})
\end{aligned}
\right.
\end{equation}
for any strictly positive integer $\alpha$.

The existence of solutions to~\eqref{eqstatsys} is an immediate consequence of the Cauchy-Lipschitz theorem. More precisely, we have the following lemma.

\begin{lem} Let $\alpha\in \N^*$ and $(x_0,\chi_0,\varphi_0)\in \R\times\R^2$. For any $a,b>0$, there exist $\gamma>0$ and $(\chi,\varphi)\in \mathcal C^1([x_0-\gamma,x_0+\gamma],\R^2)$ unique solution of~\eqref{eqstatsys} satisfying $\chi(x_0)=\chi_0$, $\varphi(x_0)=\varphi_0$. Moreover, $(\chi,\varphi)$ can be extended on a maximal interval $]X_-,X_+[$ for which the following holds
\begin{enumerate}
\item either $X_+=+\infty$ or $X_+<+\infty$ and $\lim_{x\to X_+}|\chi|+|\varphi|=+\infty$,
\item either $X_-=-\infty$ or $X_->-\infty$ and $\lim_{x\to X_-}|\chi|+|\varphi|=+\infty$.
\end{enumerate} 
\end{lem}

Our goal is then to find an initial condition $(\chi_0,\varphi_0)\in \R^2$ such that the unique solution of~\eqref{eqstatsys} satisfying $\chi(x_0)=\chi_0$, $\varphi(x_0)=\varphi_0$ is defined on $\R$ and $\lim_{x\to\pm \infty}(\chi(x),\varphi(x))=(0,0)$.

A key ingredient of the proof is to remark that system~\eqref{eqstatsys} is the Hamiltonian system associated with the energy
\begin{equation}\label{energysys}
H(\chi,\varphi)=\frac{1}{2}\chi^2(1-\varphi^{2\alpha})+\frac{a}{2(\alpha+1)}\varphi^{2\alpha+2}-\frac{b}{2}\varphi^2.
\end{equation}
As a consequence, to have a complete description of the dynamical system, it is enough to analyze the energy levels of~\eqref{energysys}, \emph{i.e.} the curves in the $(\chi,\varphi)$-plane defined by $\Gamma_c=\{(\chi, \varphi)| H(\chi,\varphi)=c\}$.

\begin{lem}\cite[for $\alpha=1$]{EstRot-12} Let $\alpha\in \N^*$. For any $a,b>0$, $H$ has the following properties:
\begin{enumerate}
\item if $a-b>0$,
\begin{enumerate}
\item $\nabla H(\chi,\varphi)=0$ if and only if $(\chi,\varphi)=(0,0)$, $(\chi,\varphi)=(0,\pm \left(\frac{b}{a}\right)^{1/2\alpha})$, $(\chi,\varphi)=(\pm\sqrt{\frac{a-b}{\alpha}},1)$ or if $(\chi,\varphi)=(\pm\sqrt{\frac{a-b}{\alpha}},-1)$;
\item $(\chi,\varphi)=(0,\pm \left(\frac{b}{a}\right)^{1/2\alpha})$ are local minima, and $(\chi,\varphi)=(0,0)$, $(\chi,\varphi)=(\pm\sqrt{\frac{a-b}{\alpha}},1)$ and $(\chi,\varphi)=(\pm\sqrt{\frac{a-b}{\alpha}},-1)$ are saddle points of the energy $H$.
\end{enumerate}
\item if $a-b=0$, $\nabla H(\chi,\varphi)=0$ if and only if $(\chi,\varphi)=(0,0)$ or $(\chi,\varphi)=(0,\pm 1)$.
\item if $a-b<0$,
\begin{enumerate}
\item $\nabla H(\chi,\varphi)=0$ if and only if $(\chi,\varphi)=(0,0)$ or $(\chi,\varphi)=(0,\pm \left(\frac{b}{a}\right)^{1/2\alpha})$;
\item $(\chi,\varphi)=(0,0)$ and $(\chi,\varphi)=(0,\pm \left(\frac{b}{a}\right)^{1/2\alpha})$ are saddle points of the energy $H$.
\end{enumerate}
\end{enumerate}

\end{lem}

\begin{rem}\label{solequiv0} Let $\alpha\in \N^*$, $a,b>0$ and $(\chi,\varphi)$ a continuous solution of~\eqref{eqstatsys}, if there exists $x_0\in \R$ such that $(\chi(x_0), \varphi(x_0))=(0,0)$ then $(\chi(x), \varphi(x))\equiv (0,0)$. 
\end{rem}

\begin{rem}\label{solequiv1} Let $\alpha\in \N^*$, $a,b>0$ and $x_0, \chi_0 \in \R$. If $(\chi,\varphi)$ is the unique continuous solution of~\eqref{eqstatsys} satisfying $\chi(x_0)=\chi_0$, $\varphi(x_0)=1$ (resp. $\varphi(x_0)=-1$) then $\varphi(x)\equiv 1$ (resp. $\varphi(x)\equiv -1$). 
\end{rem}

As a consequence of Remark~\ref{solequiv1}, we have the following lemma.

\begin{lem}\label{lemsolsmall1} Let $\alpha\in \N^*$, $a,b>0$ and $(\chi,\varphi)$ a continuous solution of~\eqref{eqstatsys} defined on $\R$ such that $\lim_{x\to \pm \infty} (\chi(x),\varphi(x))=(0,0)$. Then $\varphi^2(x)<1$ for all $x\in \R$.
\end{lem}

\begin{proof} Let $(\chi,\varphi)$ a continuous solution of~\eqref{eqstatsys}  converging to $(0,0)$ as $x\to \pm\infty$. Then $\forall\varepsilon>0$, $\exists \delta_\varepsilon>0$ such that $|\varphi(x)|<\varepsilon$ for all $x>\delta_\varepsilon$. 

Suppose, by contradiction, that there exists $x_0\in \R$ such that $\varphi(x_0)= 1$ (resp. $\varphi(x_0)= -1$). Then Remark~\ref{solequiv1} implies $\varphi(x)\equiv 1$ (resp. $\varphi(x)\equiv -1$), a contradiction.

\end{proof}

Next, since we are interested in solutions of~\eqref{eqstatsys}  converging to $(0,0)$ as $x\to \pm\infty$, we consider the zero level set of H defined by~\eqref{energysys}. The curve $\Gamma_0=\{(\chi, \varphi)| H(\chi,\varphi)=0\}$ is an algebraic curve of degree $2\alpha+2$ defined by
\begin{equation}\label{zerolevelset}
\chi^2(1-\varphi^{2\alpha})+\frac{a}{\alpha+1}\varphi^{2\alpha+2}-b\varphi^2=0
\end{equation}
and its behavior depends on the parameters $a$, $b$ and $\alpha$. Moreover in the region $\{(\chi,\varphi)\in \R^2, \varphi^2<1\}$, the equation of $\Gamma_0$ can be written as 
\begin{equation}\label{zerolevelsetbis}
\chi^2=\varphi^2\frac{b-\frac{a}{\alpha+1}\varphi^{2\alpha}}{1-\varphi^{2\alpha}}.
\end{equation}

The following two lemmas show the nonexistence of positive solutions of~\eqref{eqmodstat} that vanishes at $\pm \infty$ whenever $0<a\le (\alpha+1)b$.

\begin{lem} Let $\alpha\in \N^*$ and $a,b>0$ such that $a<(\alpha+1)b$. Let $(\chi,\varphi)$ be the solution of~\eqref{eqstatsys} satisfying $\chi(x_0)=\chi_0$, $\varphi(x_0)=\varphi_0$. If $(\chi_0, \varphi_0)\neq (0,0)$ there is no solution that satisfies $\lim_{x\to \pm \infty} (\chi(x),\varphi(x))=(0,0)$.
\end{lem}

\begin{proof} Since $\varphi$ is a $\mathcal C^1$ function such that $\lim_{x\to\pm \infty}\varphi(x)=0$, there exits $\tilde x\in \R$ such that $\varphi'(\tilde x)=0$. Moreover, thanks to Lemma~\ref{lemsolsmall1}, $\varphi^2(\tilde x)<1$. 
Hence, since $\varphi$ is a solution to~\eqref{eqstatsys}, we deduce $\chi(\tilde x)=0$ that leads to 
$$
\varphi^2(\tilde x)\frac{b-\frac{a}{\alpha+1}\varphi^{2\alpha}(\tilde x)}{1-\varphi^{2\alpha}(\tilde x)}=0
$$
using~\eqref{zerolevelsetbis}.

Now, $\varphi^{2\alpha}(\tilde x)<1<\frac{(\alpha+1)b}{a}$. Then $\varphi(\tilde x)=0$ and it follows from Remark~\eqref{solequiv0}, that $(\chi(x),\varphi(x))\equiv (0,0)$. This contradicts $(\chi_0, \varphi_0)\neq (0,0)$.

\end{proof}

\begin{lem} Let $\alpha\in \N^*$ and $a,b>0$ such that $a=(\alpha+1)b$. Let $(\chi,\varphi)$ be the solution of~\eqref{eqstatsys} satisfying $\chi(x_0)=\chi_0$, $\varphi(x_0)=\varphi_0$. If $(\chi_0, \varphi_0)\neq (0,0)$ there is no solution that satisfies $\lim_{x\to \pm \infty} (\chi(x),\varphi(x))=(0,0)$. 
\end{lem}

\begin{proof} If  $a=(\alpha+1)b$, the curve $\Gamma_0$ is defined as 
\begin{equation}\label{zerolevelsetspecial}
(1-\varphi^{2\alpha})(\chi^2-b\varphi^2)=0
\end{equation}
and zero contour line is represented in Fig.~\ref{figzerolevelset}.
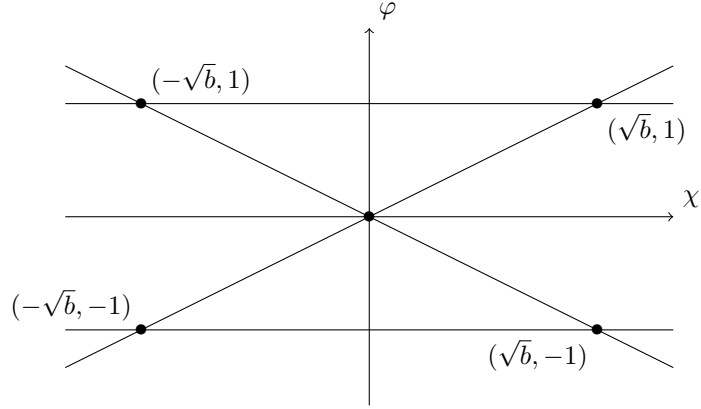
\begin{figure}[h]
\begin{center}
	\begin{tikzpicture}
	\node (O) at (0,0) {};
	\draw[->] (0,-2.5) to (0,2.5);
	\draw[->] (-4,0) to (4,0);
	\node[above right] at (4,0) {$\chi$};
	\node[above right] at (0,2.5) {$\varphi$};
	\draw[-] (-4,1.5) to (4,1.5);
	\draw[-] (-4,-1.5) to (4,-1.5);
	\draw[-] (-4,2) to (4,-2);
	\draw[-] (-4,-2) to (4,2);
	\draw (0,0) node {$\bullet$};
	\draw (-3,1.5) node {$\bullet$};
	\node[above right] at  (-3,1.5)  {$(-\sqrt{b},1)$};
	\draw (3,1.5) node {$\bullet$};
	\node[below right] at  (3,1.5)  {$(\sqrt{b},1)$};
	\draw (-3,-1.5) node {$\bullet$};
	\node[above left] at  (-3,-1.5)  {$(-\sqrt{b},-1)$};
	\draw (3,-1.5) node {$\bullet$};
	\node[below left] at  (3,-1.5)  {$(\sqrt{b},-1)$};
	\end{tikzpicture}
\end{center}
\caption{$\Gamma_0$ when $a=(\alpha+1)b$}
\label{figzerolevelset}
\end{figure}
Recall that $(\chi(x),\varphi(x))\equiv (0,0)$, $(\chi(x),\varphi(x))\equiv (\pm \sqrt{b},1)$ and $(\chi(x),\varphi(x))\equiv (\pm \sqrt{b},-1)$ are solutions to~\eqref{eqstatsys} such that $H(\chi,\varphi)=0$.

Since we are interested in solutions of~\eqref{eqstatsys}  converging to $(0,0)$ as $x\to \pm\infty$, we can restrict our study to the region $\{(\chi,\varphi)| \varphi^2<1\}$. Moreover, because of the symmetries of the problem, it is enough to look for $\varphi\ge 0$. Then, equation~\eqref{zerolevelsetspecial} implies that 
$\chi(x)=\sqrt{b}\varphi(x)$ or $\chi(x)=-\sqrt{b}\varphi(x)$. In both cases, the solution $(\chi,\varphi)$ is defined for all $x\in\R$ since $|\varphi|+|\chi|<+\infty$. Finally, since $(\chi, \varphi)$ has to converge to one of the critical points of $H$ as $x$ goes to $\pm \infty$, we can conclude that either
\begin{equation*}
\lim_{x\to -\infty} (\chi(x),\varphi(x))=(0,0) \text{ and } \lim_{x\to +\infty} (\chi(x),\varphi(x))=(\sqrt{b},1)
\end{equation*}
or
\begin{equation*}
\lim_{x\to -\infty} (\chi(x),\varphi(x))=(-\sqrt{b},1) \text{ and } \lim_{x\to -\infty} (\chi(x),\varphi(x))=(0,0).
\end{equation*}
As a consequence, the unique solution of~\eqref{eqstatsys}  that satisfies $\lim_{x\to \pm \infty} (\chi(x),\varphi(x))=(0,0)$ is $(\chi(x),\varphi(x))\equiv (0,0)$.

\end{proof}

Hence, let $a>(\alpha+1)b$. In this case we are able to prove the existence and uniqueness (modulo translations) of a positive solution $\varphi$ to of~\eqref{eqmodstat} such that $\lim_{x\to\pm\infty}\varphi(x)=0$.

\begin{lem}\label{existencesolstat} Let $\alpha\in \N^*$ and $a,b>0$ such that $a>(\alpha+1)b$. Then there is a unique (modulo translations) solution $(\chi,\varphi)$ of~\eqref{eqstatsys} that satisfies $0<\varphi(x)<1$ for all $x\in \R$ and
$\lim_{x\to \pm \infty} (\chi(x),\varphi(x))=(0,0)$. Moreover, there exists $x_0\in \R$ such that $\varphi$ is symmetric with respect to $\{x=x_0\}$ and strictly decreasing for all $x>x_0$. 
\end{lem}

\begin{proof} If $(\chi,\varphi)$ is a solution of~\eqref{eqstatsys} that satisfies $\varphi(x)>0$ for all $x\in \R$ and
$\lim_{x\to \pm \infty} (\chi(x),\varphi(x))=(0,0)$, then there exists $x_0 \in \R$ such that $\varphi'(x_0)=0$ and $0<\varphi^2(x_0)<1$. 
Hence, $\chi(x_0)=0$ and, thanks to~\eqref{zerolevelsetbis}, $\varphi(x_0)=\left(\frac{(\alpha+1)b}{a}\right)^{\frac{1}{2\alpha}}<1$.

Let $(\chi_0,\varphi_0)=(0,\left(\frac{(\alpha+1)b}{a}\right)^{\frac{1}{2\alpha}})$ and $(\chi,\varphi)$ the unique solution of~\eqref{eqstatsys} satisfying $\chi(x_0)=\chi_0$ and $\varphi(x_0)=\varphi_0$. Our goal is to show $\varphi(x)>0$ for all $x\in \R$ and $\lim_{x\to \pm \infty} (\chi(x),\varphi(x))=(0,0)$.

First of all, we show that $(\chi,\varphi)$ is defined for all $x\in \R$. Indeed, let $]X_-,X_+[$ the maximal interval on which the solution $(\chi,\varphi)$ is defined. Thanks to Remark~\ref{solequiv1}, we have $\varphi^2(x)<1$ for all $x\in ]X_-,X_+[$. This, together with~\eqref{zerolevelsetbis} and the fact that $b<\frac{a}{\alpha+1}$, implies $\chi^2(x)<\frac{a}{\alpha+1}$ for all $x\in]X_-,X_+[$. As a consequence, $]X_-,X_+[=\R$.

Next, we remark that for any $x>x_0$, $\chi(x)<0$. Indeed, equation~\eqref{eqstatsys} implies that $\chi'(x_0)=-\alpha b\varphi(x_0)<0$. Hence, $\chi(x)<0$ for all $x>x_0$ sufficiently close to $x_0$. 
Assume, by contradiction, that there is $\tilde x\in ]x_0,+\infty[$ such that $\chi(x)<0$ for all $x\in ]x_0,\tilde x[$ and $\chi(\tilde x)=0$. As a consequence, $\varphi'(x)<0$ for all $x\in ]x_0,\tilde x[$ and 
$$
\varphi^2(\tilde x)\frac{b-\frac{a}{\alpha+1}\varphi^{2\alpha}(\tilde x)}{1-\varphi^{2\alpha}(\tilde x)}=0.
$$
This implies $\varphi(\tilde x)=0$, a contradiction.

With the same argument, we show $\chi(x)>0$ for any $x<x_0$ and we conclude that $\varphi(x)>0$ for all $x\in\R$. Moreover, $\varphi$ is strictly decreasing for all $x>x_0$ and strictly increasing for all $x<x_0$.
As a consequence $\max_{x\in\R} \varphi(x)=\varphi(x_0)=\left(\frac{(\alpha+1)b}{a}\right)^{\frac{1}{2\alpha}}$. 

Finally, we know that $(\chi, \varphi)$ has to converge to one of the critical points of $H$ as $x$ goes to $\pm \infty$. On the one hand, the points $(\pm\sqrt{\frac{a-b}{\alpha}},1)$ and $(\pm\sqrt{\frac{a-b}{\alpha}},-1)$ are excluded since $0\le \lim_{x\to\pm \infty}\varphi(x)\le \varphi(x_0)<1$. On the other hand, $H(0,\pm \left(\frac{b}{a}\right)^{1/2\alpha})<0=H(\chi_0,\varphi_0)$. Hence $(\chi, \varphi)$ has to converge to $(0,0)$ as $x$ goes to $\pm \infty$.

Next, we have to prove that $\varphi$ is symmetric with respect to $\{x=x_0\}$, \emph{i.e} $\varphi(2x_0-x)=\varphi(x)$ for all $x\in \R$. For this, it is enough to remark that $(\chi_{x_0}(x),\varphi_{x_0}(x)):=(-\chi(2x_0-x),\varphi(2x_0-x))$ is a solution to~\eqref{eqstatsys} that satisfies $\chi_{x_0}(x_0)=0$ and $\varphi_{x_0}(x_0)=\varphi_0$. As a consequence $(\chi_{x_0}(x),\varphi_{x_0}(x))=(\chi(x),\varphi(x))$ for all $x\in \R$.


\end{proof}

\begin{rem}  As shown in Section~\ref{secgroundstates}, a straightforward computation leads to the explicit formula for $\varphi$. In particular, for any fixed $x_0\in \R$, 
\begin{equation}\label{solstatexplicit}
\varphi(x)=\left(\frac{1}{2}\left(\frac{a}{(\alpha+1)b}+1\right)+\frac{1}{2}\left(\frac{a}{(\alpha+1)b}-1\right)\cosh({2\alpha\sqrt{b}(x-x_0)})\right)^{-\frac{1}{2\alpha}}
\end{equation}
for $x\ge x_0$ and $\varphi(x)=\varphi(2x_0-x)$ for $x<x_0$.
\end{rem}

Finally, we conclude by proving the non-degeneracy of the solution $\varphi$. The linearized operator at our solution $\varphi$ is defined by 
\begin{align}\label{linearizedphi}
Lv=&-\left(\frac{v'}{1-\varphi^{2\alpha}}+\left(\frac{\varphi'}{1-\varphi^{2\alpha}}\right)\frac{2\alpha\varphi^{2\alpha-1}v}{1-\varphi^{2\alpha}}\right)'\nonumber\\
&+2\alpha\left(\frac{\varphi'}{1-\varphi^{2\alpha}}\right)\varphi^{2\alpha-1}\left(\frac{v'}{1-\varphi^{2\alpha}}+\left(\frac{\varphi'}{1-\varphi^{2\alpha}}\right)\frac{2\alpha\varphi^{2\alpha-1}v}{1-\varphi^{2\alpha}}\right)\nonumber\\
&+\alpha(2\alpha-1)\left(\frac{\varphi'}{1-\varphi^{2\alpha}}\right)^2\varphi^{2\alpha-2}v-a(2\alpha+1)\varphi^{2\alpha}v+b v.
\end{align}
Our goal is to prove that in $L^2$, $\ker L=\mathrm{span}\{\varphi'\}$.

\begin{lem}\label{nondegeneracy} Let $\alpha\in \N^*$ and $a,b>0$ such that $a>(\alpha+1)b$ and $\varphi$ the unique solution of~\eqref{eqmodstat} that satisfies $0<\varphi(x)<1$ for all $x\in \R$ and
$\lim_{x\to \pm \infty} \varphi(x)=0$. Then $\varphi$ is non-degenerate, \emph{i.e.} $\ker L=\mathrm{span}\{\varphi'\}$.
\end{lem}

\begin{proof} First of all, by deriving the equation~\eqref{eqmodstat} with respect to $x$, we can easily remark that $\varphi'$ is a solution to $Lv=0$. As a consequence $\mathrm{span}\{\varphi'\}\subset \ker L$.

Let $v\in \ker L\smallsetminus \mathrm{span}\{\varphi'\}$. As a consequence, by writing $u=\frac{v'}{1-\varphi^{2\alpha}}+\left(\frac{\varphi'}{1-\varphi^{2\alpha}}\right)\frac{2\alpha\varphi^{2\alpha-1}v}{1-\varphi^{2\alpha}}$, we deduce that $v$ is a solution to

\begin{equation}\label{eqstatsyslin}
\left\{
\begin{aligned}
&u'=2\alpha\chi\varphi^{2\alpha-1}u+\alpha(2\alpha-1) \chi^2\varphi^{2\alpha-2}v-a(2\alpha+1)\varphi^{2\alpha}v+b v\\
&v'= (1-\varphi^{2\alpha})u-2\alpha\chi\varphi^{2\alpha-1}v
\end{aligned}
\right.
\end{equation}
with $\chi=\frac{\varphi'}{1-\varphi^{2\alpha}}$. Note 
that~\eqref{eqstatsyslin} is simply the linearization 
of~\eqref{eqstatsys} at the solution $(\chi,\varphi)$, 
$(\chi',\varphi')$ is a solution to~\eqref{eqstatsyslin} and the 
non-degeneracy property will follow from the fact $(0,0)$ is a non-degenerate minimum of $H$. More precisely, a direct computation, using the fact that $(u,v)$ and $(\chi',\varphi')$ are solutions to~\eqref{eqstatsyslin}, gives $(u\varphi' - v\chi')'=0$. Hence, the function $u\varphi' - v\chi'$ is constant and in particular 
$u(x)\varphi'(x) - v(x)\chi'(x)=u(x_0)\varphi'(x_0) - v(x_0)\chi'(x_0)=  - v(x_0)\chi'(x_0)$ with $x_0$ defined as in the proof of Lemma~\ref{existencesolstat} so that $\chi'(x_0)=-\alpha b\varphi(x_0)<0$.
If $v(x_0)=0$, then $u\varphi' - v\chi'\equiv 0$ and using the definition of $u$ we can conclude $v\equiv 0$.
If $v(x_0)>0$ (resp. $v(x_0)<0$), then $u(x)\varphi'(x) - v(x)\chi'(x)=\alpha b\varphi(x_0)v(x_0)>0$ (resp. $u(x)\varphi'(x) - v(x)\chi'(x)=\alpha b\varphi(x_0)v(x_0)<0$) and $(u,v)$ cannot converge to $(0,0)$ as $x$ goes to $\pm\infty$. Hence, $v\notin \ker L$.

\end{proof}

\end{appendix}


\end{document}